\documentclass[11pt,english,a4paper]{amsart} 

\usepackage{amsthm}
\usepackage{amsmath}
\usepackage{amssymb}
\usepackage{xcolor}

\usepackage{todonotes}

\usepackage{a4wide}
\usepackage{mathrsfs}

\usepackage{graphicx}
\usepackage{caption}
\usepackage{subcaption}
\usepackage{color}
\usepackage{a4wide}
\usepackage{amscd}
\usepackage{latexsym}
\usepackage{dsfont}
\usepackage{etoolbox} 
\usepackage{xparse} 
\usepackage{mathtools}

\usepackage{hyperref}
\hypersetup{
	colorlinks,
	linkcolor=red,
	citecolor=black
}
\usepackage{colonequals}
\usepackage{siunitx}

\usepackage{stmaryrd}
\usepackage{tikz}
\usetikzlibrary{shapes.geometric, arrows.meta, decorations.pathmorphing}
\definecolor{black}  {RGB}{  0, 100, 200}
\definecolor{Green} {RGB}{150, 200,   0}
\definecolor{Gray}  {RGB}{150, 150, 150}
\definecolor{Orange}{RGB}{250, 150,   0} 
\definecolor{Red}   {RGB}{200,  50,  50} 

\usepackage{algorithm}
\usepackage{algpseudocode}

\DeclareMathOperator*{\argmin}{argmin}

\global\long\def\Cc{\mathbb{C}}

\global\long\def\R{\mathbb{R}}
\global\long\def\N{\mathbb{N}}

\global\long\def\cJ{\mathcal{J}}
\global\long\def\cJb{J}
\global\long\def\cG{\mathcal{G}}

\global\long\def\cO{\mathcal{O}}

\global\long\def\Tt{\mathbf{T}}
\global\long\def\Ut{\mathbf{U}}
\global\long\def\At{\mathbf{A}}
\global\long\def\Qt{\mathbf{Q}}
\global\long\def\Vt{\mathbf{V}}
\global\long\def\Psit{\mathbf{\Psi}}

\global\long\def\It{\mathbf{I}}
\global\long\def\bt{\mathbf{b}}

\global\long\def\ord{p}

\global\long\def\s{t}

\newcommand{\core}[1]{\left\llbracket\, \begin{matrix} #1 \end{matrix} \,\right\rrbracket}

 \theoremstyle{plain}
   \newtheorem{counter}{\protect\examplename}[section]

 \theoremstyle{plain}
   \newtheorem{example}[counter]{\protect Example}
  
 \theoremstyle{plain}

 \theoremstyle{plain}

 \theoremstyle{plain}
 \newtheorem{remark}[counter]{\protect Remark} 
  
\theoremstyle{plain}
  \newtheorem{problem}[counter]{\protect Problem}
  
 \theoremstyle{plain}

 \theoremstyle{plain}
   \newtheorem{proposition}[counter]{\protect Proposition}
  
 \theoremstyle{plain}

  \theoremstyle{plain}

\definecolor{grey}{rgb}{0.5,0.5,0.5}

\begin{document}

 \title[Fredholm integral equations for approximation and  training]{Fredholm integral equations for function approximation and  the training of neural networks}

\author{Patrick Gel\ss , Aizhan Issagali, and Ralf Kornhuber}

\address{Patrick Gel\ss, 
	AI in Society, Science, and Technology, 
 Zuse Institute Berlin, 
 14195 Berlin, Germany}
\email{gelss@zib.de}

\address{Aizhan Issagali, Institut f\"ur Mathematik, 
 Freie Universit\"at Berlin, 
 14195 Berlin, Germany}
\email{aizhan.issagali@fu-berlin.de}

\address{Ralf Kornhuber, Institut f\"ur Mathematik, 
 Freie Universit\"at Berlin, 
 14195 Berlin, Germany}
\email{kornhuber@math.fu-berlin.de}

\begin{abstract}
We present a novel and mathematically transparent approach to function approximation and the  training of large, high-dimensional neural networks, 
based on the approximate least-squares solution of associated Fredholm integral equations of the first kind by  Ritz-Galerkin discretization, 
Tikhonov regularization and tensor-train methods.
Practical application to supervised learning problems of regression and classification type confirm that the resulting algorithms are competitive with state-of-the-art neural network-based methods.
\end{abstract}


\thanks{
This work was supported by the QuantERA II programme funded by the European Union’s Horizon 2020 research and innovation programme under Grant Agreement No. 101017733 and by the research center MATH+.\\
We would like to thank the anonymous referees for their constructive remarks that considerably helped to improve the paper.
}
 
\maketitle

%
 \section{Introduction}
Efficient and reliable methods for the training of neural networks are of fundamental importance for a multitude of applications
and therefore a flourishing field of mathematical research.
Stochastic gradient methods~\cite{amari1993backpropagation,bottou2012stochastic} are known and further developed since the sixties.
These methods  mostly perform very well and therefore are often considered as methods of choice in the field.
However, global convergence analysis is complicated and actual convergence might depend on  proper parameter tuning~\cite{cheridito2021non}.
Multilevel trust region methods~\cite{chen2022randomized,conn2000trust,mizutani2003structure}  on the one hand aim at increasing  efficiency 
by Newton-type linearization and on the other hand at increasing reliability by suitable step size control.
Multilevel versions~\cite{kopanicakova2022globally}  are intended to additionally transfer the 
efficiency of 
multigrid methods for the minimization of  energy functionals associated with
elliptic partial differential equations to  the minimization of loss functionals.

This paper is devoted to a novel and mathematically transparent approach to function approximation and
the training of large and high-dimensional shallow neural networks
with usual kernel functions $\psi$, training parameters $(\eta_i,u_i)_{i=1}^N$ with
$\eta_i \in Q \subset \R^s$, $u_i \in \R$, $I=1, \dots,N$, and intentionally large $s$ and $N$. 
The starting point of our considerations is that such neural networks can be interpreted
as Monte-Carlo approximations of  integrals  over $Q$,
once the parameters  $\eta_i$ are assumed to be equidistributed  on $Q$ (see also \cite{rotskoff2019trainability}).
Resulting Fredholm integral operators, now acting on  parameter functions rather than parameter vectors,
give rise to  so-called {\em Fredholm networks} providing a novel  approach to function approximation.
Training of such Fredholm networks amounts to  least-squares solution of linear Fredholm integral equations of the first kind
for a parameter function $u \colon Q \to \R$, the continuous counterpart of the discrete parameter vector $(u_i)_{i=1}^N$.
Fredholm-trained parameters  $ (\eta_i, u(\eta_i))_{i=1}^N$ of the original discrete neural network are then obtained by suitable sampling of $\eta_i \in Q$.

A similar approach has been developed for kernel methods in supervised and semi-supervised learning~\cite{que2016back,que2014learning,wang2019semi} 
and applied to the covariance shift problem in transfer learning~\cite{que2013inverse}.
%
 {Recent mean-field neural networks~\cite{Nitanda2022mfld} can be regarded as  Fredholm networks on measures. 
The associated mean-field approach for the training of neural networks 
is based on stochastic gradient flows for the minimization of corresponding (regularized) loss functionals on the space of probability measures.  
It provides convergence results for well-known stochastic gradient descent~\cite{mei2018mean,rotskoff2018neural,Hu2021mfld} and suitable generalizations~\cite{chizat2018global,Chizat2022mfld}.
}

%
%
 {In this paper, the continuous Fredholm training problem is discretized by Ritz-Galerkin methods 
with suitable ansatz spaces, and we apply Tikhonov regularization~\cite{Engl}.
Instead of gradient descent, the resulting (intentionally large) linear algebraic systems are then solved approximately by tensor-train methods, 
such as alternating linear schemes~\cite{HOLTZ2012,KLUS2019},
in order to enhance efficiency and to allow for more general architectures.
}
Approximately Fredholm-trained parameters  $(\tilde{\eta}_i,\tilde{u}(\tilde{\eta}_i))_{i=1}^N$ of the original  neural network 
are finally obtained from the resulting approximation $\tilde{u}$ of $u$ and  suitable sampling of $\tilde{\eta}_i \in Q$.
 {A priori residual error estimates for the loss functional, both for trained Fredholm and neural networks,  are given
in the special case of piecewise constant approximation and exact linear solution.}
While a  {comprehensive} numerical analysis   {and  possible extensions to deep neural networks}
are devoted to forthcoming papers, 
 {a major} goal of this work is to confirm the practical potential of the Fredholm approach by numerical experiments 
and comparison with state-of-the-art  neural network-based methods.

The paper is organized as follows. After describing the architecture of the neural networks in question
together with the resulting discrete training problem in the next Section~\ref{sec:NN},
we introduce Fredholm networks and Fredholm training problems
as  their continuous counterparts and briefly discuss  ill-posedness in Section~\ref{sec:AL}. 
Ritz-Galerkin discretization and Tikhonov regularization are presented in Section~\ref{sec:R-GT}.
Utilizing suitable factorizations of the kernel functions $\psi$ (as derived in the Appendix~\ref{app}), 
Section~\ref{sec:TT} is devoted to tensor-train methods for the iterative solution of the  large linear systems 
associated with intentionally high-dimensional parameter domains $Q\subset \R^s$.
In Section~\ref{sec:SAMPLE}, we harvest the results on the training of  Fredholm networks, briefly discuss Monte-Carlo sampling,
and introduce a novel inductive importance sampling strategy. 
In the next Section~\ref{sec:NUMEX}, we apply our methods to three well-established test problems of regression and classification type
with different computational complexity. In particular, we consider the UCI banknote authentication data set from \cite{Lohweg2013},
concrete compressive strength prediction~\cite{YEH1998, Yeh2003, Yeh2006, Shah2022}, and the well-known MNIST data set~\cite{Lecun1998}. 
For all these problems, Fredholm networks  and Fredholm-trained neural networks
turn out to be highly competitive  with state-of-the-art  approaches without any problem-specific features or tuning.
 {
Note that, with some linear algebraic tricks, neural tangent kernel (NTK) methods~\cite{jacot2018neural} are able to achieve
comparable performance even on more challenging datasets like CIFAR-10~\cite{arora2019harnessing}.}
The efficiency of Fredholm-trained neural networks however seems to rely  on a careful choice of sampling strategy:
While straightforward quasi-Monte-Carlo sampling appears to work perfectly well  for simple problems, 
more advanced strategies, such as inductive importance sampling, seem to be required in more complicated cases. 
 {
The final Section~\ref{sec:CONCOUT} is devoted to some concluding remarks and a brief outlook on deep Fredholm networks.}

%
%
\section{Large single layer neural networks}\label{sec:NN}
\subsection{Architecture}\label{sec:ARCH}
We consider a parametrized single layer neural network of the form 
\begin{equation} \label{eq:PARREP}
X \ni x \mapsto F_N(x, \zeta) = \frac{1}{N} \sum_{i=1}^N   \psi(x, \eta_i) u_i \in \R
\end{equation}
defined on a compact hypercube $X \subset \R^d$, for simplicity, 
with intentionally large number of neurons $N \in \N$, a suitable kernel function $ \psi: X \times Q \to \R$,
and associated parameter vector
\[
\zeta = (\zeta_i)_{i=1}^N, \quad  \zeta_i = (\eta_i, u_i) \in Q \times \R,
\]
where $Q = Q_1 \times \cdots \times Q_s\subset \R^s$. 

\emph{Ridge kernel functions} 
\begin{equation} \label{eq:RIGDE}
 \psi(x, \eta ) =  \sigma( w \cdot x + b), \quad \eta = (w, b) \in Q \subset \R^s, \; s=d+1,
\end{equation}
are characterized by suitable activation functions  $\sigma: \R \to \R$
and give rise to so-called single hidden layer perceptron models~\cite{pinkus1999approximation}.
In the simplest case  $\sigma \equiv \text{Id}$ and $b=0$, this definition reduces to the Euclidean scalar product  
$\psi(x, \eta) =  x \cdot \eta$ with $\eta = w \in \R^s$ and $s=d$.
More relevant choices for $\sigma$ include  the (discontinuous) \emph{Heaviside function}
\begin{equation} \label{eq:HEAVI}
\sigma_{\infty}(z) = \left \{
\begin{array}{ccl}
0 & \text{ if } z < 0,\\
1/2 & \text{ if } z =0, \\
1 & \text{ if } z > 0\\
\end{array}
\right .
\end{equation}
or suitable regularizations like the \emph{logistic activation function}   
\begin{equation} \label{eq:LOGISTIC}
 \sigma_\kappa(z) = \frac{\exp(\kappa z)}{1 + \exp(\kappa  z)} 
\end{equation}
with  $\kappa >0$. 
Note that $\sigma_{\kappa}(z) \to \sigma_{\infty}(z)$ holds for $\kappa \to \infty$ and all $z\in \R$. 

\emph{Radial kernel functions}~\cite{KARLIK2011} are  given by
\begin{equation} \label{eq:RADIAL}
\psi(x, \eta) = \sigma( \lVert x - w  \rVert ), \quad \eta = w \in Q\subset \R^s ,\; s=d,
\end{equation}
with  some vector norm $\lVert\,\cdot\,\rVert$ on $\R^d$ and suitable activation functions $\sigma$.
In the case of Gaussian activation functions 
\begin{equation} \label{eq:RADAC}
\sigma(z) = \exp\left( - \frac{z^2}{2 \kappa^2} \right)
\end{equation}
and the Euclidean norm $\lVert\,\cdot\,\rVert=\lVert\,\cdot\,\rVert_2$, we have
\begin{equation} \label{eq:GAUSSAC}
\psi(x,\eta) = 
\exp\left( - \frac{\lVert x- \eta \rVert_2^2 }{2 \kappa^2}\right) = 
\prod_{k=1}^d \exp \left( - \frac{(x_k-\eta_k)^2}{2 \kappa^2} \right).
\end{equation}
\subsection{Function data, loss functional, and  training} \label{subsec:FUNCDLO}
Throughout the following 
$L^2(X,\pi)$ stands for the Hilbert space of $\pi$ measurable, 
quadratically integrable functions on $X$ with canonical scalar product $(\cdot,\cdot)_\pi$ 
and induced norm $\| \cdot \|_\pi$. We write $L^2(X) = L^2(X,\pi)$ and $(\cdot,\cdot)_X=(\cdot,\cdot)_\pi$, $\|\cdot \|_X=\| \cdot \|_\pi$,
if $d\pi(x) =dx$ is the Lebesgue measure. Furthermore, $L^2(Q)$ stands for the  Hilbert space of Lebesgue measurable, 
quadratically integrable functions on $Q$ with canonical scalar product $(\cdot,\cdot)$ 
and induced norm $\| \cdot \|$.

We aim at the approximation of a function  $F: X \mapsto \R$
by parametrized ansatz functions $F_N(\cdot, \zeta)$ of the form \eqref{eq:PARREP}.
The parameters $\zeta$ are  determined by minimizing the loss functional
\begin{equation}\label{pval}
\cJ_N(\zeta) = \|F - F_N(\cdot, \zeta)\|_\pi^2 = \int_X (F(x) - F_N(x, \zeta))^2 \; d\pi(x)
\end{equation}
where the measure $\pi$ is associated with the training data available.
We will consider the following two cases of given
\begin{equation}\label{eq:DATA}
\text{ a)  function data}\;\; F \in L^2(X) \quad 
\text{ b)  point data} \;\; ((x_j, F(x_j)))_{j=1}^M \in (X \times \R)^M \subset \R^{M(d+1)}.
\end{equation}
The case a) of a given function $F$ is associated with the Lebesgue measure $d \pi(x)= dx$
while  case b) of given point data corresponds to the point measure $d \pi(x)= d\mu(x)$,
\begin{equation}\label{point}
\mu =   \frac{1}{M}\sum_{j=1}^M \delta_{x_j}, \quad x_1, \dots, x_M\in X.
\end{equation}
Training of the neural network $F_N(\cdot, \zeta)$  amounts to the (approximate) solution of the  following  discrete, non-convex, large-scale minimization problem.
\begin{problem} \label{prob:DISCOPT}
Find $\zeta^* =  (\zeta_i^*)_{i=1}^N, \; \zeta_i^*=(\eta^*_i, u^*_i) \in Q \times \R$, such that
\begin{equation} \label{eq:DMIN}
\cJ_N(\zeta^*) \leq  \cJ_N(\zeta) \qquad \forall  \zeta=(\zeta_i)_{i=1}^N, \; \zeta_i=(\eta_i, u_i) \in Q \times \R.
\end{equation}
\end{problem}

Approximation properties of neural networks of the form \eqref{eq:PARREP} have quite a history (see, e.g., \cite{chen1995approximation}
or the overview of Pinkus~ \cite{pinkus1999approximation}).
In particular, ridge kernels with any non-polynomial activation function $\sigma$  and unconstrained parameters $\zeta \in \R^{N(d+2)}$
provide uniform approximation of continuous functions $F$  on compact sets $X \subset \R^d$~\cite[Theorem~3.1]{pinkus1999approximation}.
While similar results for constrained parameters $\zeta \in Q \times \R$ seem to be rare,
the following proposition  is a consequence of  \cite[Theorem~2.6]{stinchcombe1990approximating}.

\begin{proposition} \label{prop:DISCAPPROX}
For the ridge kernel \eqref{eq:RIGDE} with logistic activation function \eqref{eq:LOGISTIC}, 
$Q$ containing the unit ball in $\R^{d+1}$,
and any function $F \in C(X)$, we have
\begin{equation} \label{eq:NNDENSE}
\underset{\zeta \in (Q \times \R)^N}{\inf} \; \underset{x \in X}{\max} \; |F(x) - F_N(x, \zeta)| \to 0 \quad \text{for} \quad N \to \infty.
\end{equation}
\end{proposition}
Hence,  assuming existence of a solution $\zeta^*$ for all $N \in \N$, we have $\cJ_N(\zeta^*) \to 0$ as $N\to \infty$
for $F \in L^2(X)$ in the case (9a) and $F\in C(X)$ in the case (9b), respectively.

%
\section{Continuous approximations of discrete training problems}\label{sec:AL} 
Following the ideas from \cite[Section 1.2]{rotskoff2019trainability}, we now introduce and investigate the asymptotic limit of $F_N(\cdot, \zeta)$ 
and of the corresponding Training Problem~\ref{prob:DISCOPT} for $N\to \infty$.
\subsection{Asymptotic integral operators} \label{subsec:ASIOP}
Assume  that $\psi(x, \cdot)\in L^2(Q)$ for each fixed $x \in X$.
Let  $\eta_i $,  $i = 1, ..., N$, be independent, identically equidistributed 
 {random samples  and let $u_i = u(\eta_i)$,   $i = 1, ..., N$, be the values of an arbitrary function $u \in L^2(Q)$ at those sample points.}
Then, denoting $|Q|= \int_Q 1\; dx$,
\begin{equation} \label{eq:SLLN}
\frac{1}{N} \sum_{i=1}^N   \psi(x,\eta_i)  {u_i}\; \to \;   \frac{1}{|Q|}  \int_{Q} \psi(x,\eta) {u}(\eta) \; d\eta
\quad \text{for} \quad N \to \infty
\end{equation} 
holds  for  
each fixed $x \in X$ by the strong law of large numbers. 

This observation suggests to consider the Fredholm integral operator 
\begin{equation}\label{eq:QOP}
\cG v =  \frac{1}{|Q|} \int_Q \psi(\cdot, \eta)v(\eta)\; d \eta, \quad v \in L^2(Q),
\end{equation}
Let us recall some well-known properties of $\cG$, cf.~\cite{Engl,Kress,Groetsch}.
\begin{proposition}\label{prop:BOUNDCOMP}
Assume that 
\begin{equation} \label{eq:BND}
\int_X \int_Q  \psi(x,\eta)^2\; d\eta \; dx < \infty  .
\end{equation}
Then  the integral operator
\[
\cG: L^2(Q) \to L^2(X) 
\]
 is a compact, linear mapping.
\end{proposition}
Note that Proposition~\ref{prop:BOUNDCOMP} provides compactness of $\cG$,
e.g., for the radial and ridge kernel functions mentioned in Section~\ref{sec:ARCH}.

\subsection{Fredholm integral equations}
\subsubsection{Function data}
Let us first consider the case (\ref{eq:DATA}a) of a given function $F\in L^2(X)$.
In the light of \eqref{eq:SLLN}, we impose the additional constraints 
on the unknowns  $\zeta_i=(\eta_i,u_i)$, $i=1,\dots,N$, of the Training Problem~\ref{prob:DISCOPT} that
\begin{equation} \label{eq:ADHOC}
  \eta_i \text{ are samples of i.i.~equidistributed random variables and }  u_i = u(\eta_i), \quad i=1,\dots,N,
\end{equation}
with some function $u\in L^2(Q)$.
Then, for given function $F$, the asymptotic coincidence \eqref{eq:SLLN} of the discrete neural network $F_N(\cdot,( (\eta_i,u_i)_{i=1}^N)$ 
and the  integral operator $\cG u$ suggests to consider the loss functional
\begin{equation}
\cJ(v) = \int_{X}(F(x) - \cG v(x))^2\; dx.
\end{equation}
and the following continuous quadratic approximation of the discrete non-convex Training Problem~\ref{prob:DISCOPT}.  
\begin{problem}\label{prob:FFUNC}
Find $u \in L^2(Q)$ such that 
\begin{equation} \label{eq:CONTMIN}
\cJ(u) \leq \cJ(v) \quad  \forall v\in L^2(Q).
\end{equation}
\end{problem}

Observe that Problem~\ref{prob:FFUNC} is just the least-squares formulation of the Fredholm integral equation of the first kind
\begin{equation}\label{fkind}
\cG u = \frac{1}{|Q| }\int_{Q} \psi(\cdot, \eta) u(\eta)\; d\eta = F.
\end{equation}
In particular, \eqref{eq:CONTMIN} can be equivalently rewritten  as the  
normal equation
\begin{equation} \label{eq:LSQVOP}
\cG^* \cG u = \cG^* F\quad \text{in} \quad L^2(Q)
\end{equation}
utilizing  the variational formulation of \eqref{eq:CONTMIN},
\begin{equation} \label{eq:LSQVAR}
 (\cG u,\cG v)_X = (F, \cG v)_X \qquad \forall v \in L^2(Q), 
 \end{equation}
and the adjoint operator of $\cG$,
\[
L^2(X)\ni v \mapsto  \cG^* v = \frac{1}{|Q|}  \int_X \psi(x, \cdot)v(x)\; dx \in L^2(Q).
\]
This leads to the following well-known existence result (see, e.g.,  \cite[Theorem~2.5]{Engl}),
where $R(\cG)$ and $N(\cG)$ stand for the range and null space of $\cG$, respectively, $\oplus$ denotes the direct sum, 
and the superscript $\perp$ indicates the orthogonal complement in $L^2(X)$.
\begin{proposition} \label{prop:FPEXIST}
Assume that $F \in  R(\cG) + R(\cG)^\perp \subset L^2(X)$. 
Then (\ref{fkind}) has least-squares solutions $u^\dagger \oplus N(\cG)$,
with $u^\dagger = \cG^\dagger F$ being the unique minimal-norm least squares solution (best-approximate solution), 
and $\cG^\dagger \colon R(\cG) \oplus R(\cG)^\perp \to N(\cG)^\perp \subset L^2(Q)$
denoting the generalised Moore-Penrose inverse of $\cG$.
\end{proposition}
If $\cG$ is compact, then $R(\cG)$ is closed or, equivalently, $\cG^\dagger$ is bounded, 
if and only if $R(\cG)$ has finite dimension (see, e.g., \cite[Proposition~2.7]{Engl}).
As a consequence,  the integral equation (\ref{fkind}) is ill-posed, even in the least-squares sense
\eqref{eq:CONTMIN} for usual kernel functions such as the radial and ridge kernels 
presented in Section~\ref{sec:ARCH}.

\begin{proposition} \label{prop:DISCRIMINATORY}
Assume that the kernel function $\psi$ is discriminatory in the sense that
\begin{equation} \label{eq:DISCRIMINATORY}
N(\cG^*)=\{0\}.
\end{equation}
Then $R(\cG)$ is dense in $L^2(X)$ and therefore
$\underset{u \in L^2(Q)}\inf {\mathcal J}(u) = 0$ holds for all $F \in L^2(X)$.
\end{proposition}

For ridge kernels with Heaviside or logistic activation function,
we obtain $v = 0$, if $\cG^* v(\eta) = 0$  holds for all $\eta \in \R^s$ (and not only for all $\eta \in Q$) by  \cite[Lemma~1]{Cybenko}.
However, similar results for the present case of  bounded parameters $\eta \in Q$ seem to be  unknown.


\subsubsection{Point data} \label{subsec:PD}
In the case  (\ref{eq:DATA}b) of given point values of $F$, we replace the corresponding point values of $F_N$ in \eqref{pval} 
by the point values of the limiting integral operator $\cG$
to obtain the approximation $\cJb$,
\begin{equation} \label{eq:COLO}
\cJb(v) = \frac{1}{M}\sum_{j=1}^M (F(x_j) - \cG v(x_j))^2 ,
\end{equation}
of the discrete loss functional $\cJ_N$ and the following quadratic  approximation of the Training Problem~\ref{prob:DISCOPT}.  

\begin{problem}\label{prob:PPOINT}
Find $u \in L^2(Q)$ such that 
\begin{equation}\label{lq}
\cJb(u) \leq \cJb(v) \quad \forall v \in L^2(Q).
\end{equation}
\end{problem}

Problem~\ref{prob:PPOINT} is the least-squares formulation of the interpolation problem to find  $u \in L^2(Q)$ such that 
\[
\cG_M u = F_M 
\]
with $F_M = (F(x_j))_{j=1}^M \in \R^M$ and  $\cG_M \colon L^2(Q) \to \R^M$ defined by $\cG_M v = (\cG v(x_i))_{i=1}^M \in \R^M$.
 {
 As a  semi-discrete counterpart of \eqref{eq:BND}, the condition  
\begin{equation} \label{eq:BNDDI}
\frac{1}{M}\sum_{i=1}^M \int_Q \psi(x_i, \eta)^2d\eta \leq C < \infty
\end{equation}
with a constant $C$ independent of $M$ implies that $\cG_M$ is uniformly bounded in $M$.
It also holds for all kernel functions mentioned in Section~\ref{sec:ARCH}.
}

In analogy to \eqref{eq:LSQVOP}, Problem~\ref{prob:PPOINT} can be equivalently reformulated as a linear normal equation
$\cG_M^* \cG_M u = \cG_M^* F_M$ in $L^2(Q)$ utilizing  the variational formulation
\begin{equation} \label{eq:LSQVARPOINT}
 \sum_{j=1}^M(\cG u)(x_j) (\cG v)(x_j) = \sum_{j=1}^M F(x_j)(\cG v)(x_j) \qquad \forall v \in L^2(Q), 
 \end{equation}
of Problem~\ref{prob:PPOINT} with the adjoint operator $\cG_M^*$ of $\cG_M$ given by
\[
\R^M \ni v=(v_j)_{j=1}^M \mapsto \cG_M^* v = \frac{1}{|Q|} \sum_{j=1}^M \psi(x_j,\cdot ) v_j \in L^2(Q).
\]
Observe that $R(\cG_M)$ is now finite-dimensional and thus closed. 
Hence, existence of a best-approximate solution $u=\cG_M^\dagger F_M$ of Problem~\ref{prob:PPOINT} for any point data $(x_i, F(x_i))$, $i=1,\dots,M$, 
follows from $R(\cG_M) \oplus R(\cG_M)^\perp = \R^M$ and the general Theorem~2.5 in \cite{Engl}. 
As a consequence, the Moore-Penrose inverse $\cG_M^\dagger$ is now bounded.
However, its norm is expected to increase with increasing~$M$. 
\begin{remark}
Assuming  $N(\cG_M^*)=\{0\}$, we obtain $\cJb(u)=0$ in analogy to Proposition~\ref{prop:DISCRIMINATORY}.
\end{remark}
\subsection{Fredholm networks} \label{subsec:FN}
The asymptotic coincidence with  discrete neural networks $F_N(\cdot, \eta)$ suggests to  consider their continuous counterpart
\begin{equation}\label{eq:FN}
\cG v = \frac{1}{|Q|}\int_Q \psi(\cdot,\eta) v(\eta) \; d\eta
\end{equation}
with  parameter function $v \in L^2(Q)$ directly for approximation.   
Observe that the  continuous Training Problems   \ref{prob:FFUNC} (function data) or \ref{prob:PPOINT} (point data)
associated with such  \emph{Fredholm networks} are linear,  in contrast to the original discrete Problem~\ref{prob:DISCOPT}.
However, linearity comes with the price of infinitely many unknown parameter function values.
This suggests suitable discretizations of the continuous Training Problem~\ref{prob:FFUNC} or \ref{prob:PPOINT} 
providing  computable approximations of  the corresponding Fredholm networks.
For example,  the integral operator $\cG$ has been replaced by a suitable Riemann sum in \cite{que2016back,que2014learning,wang2019semi}.

 {
\begin{remark}\label{rem:MEASURE}
In order to relax the constraints \eqref{eq:ADHOC}
on the parameters $\zeta$ that provide   an asymptotic  limit of $F_N(\cdot,\zeta)$ of the form \eqref{eq:SLLN},
one might assume  that the parameters $\zeta_i = (\eta_i, u_i)$ are independent random variables that are identically distributed 
with respect to some unknown  probability measure $\nu \in \mathbb{P}(Q\times \R)$.
Then, the law of large numbers implies
\[
\frac{1}{N} \sum_{i=1}^N   \psi (x,\eta_i)u_i \;  \to \;  \int_{Q \times \R} \psi(x,\eta)u \;d\nu(\eta,u) \quad \text{for } N \longrightarrow \infty
\]
in analogy to \eqref{eq:SLLN}. Replacing $F_N(x, \zeta)$ in the loss functional $\cJ_N$ by the  mean-field network~\cite{Nitanda2022mfld}
\[
 \mathbb{P}(Q\times \R) \ni \nu \mapsto \int_{Q\times \R} \psi(\cdot,\eta) u \; d \nu(\eta,u) \in \R,
\]
the resulting training problem for the  measure $\nu$ amounts to 
the least squares formulation of  the problem to find $\nu \in \mathbb{P}(Q \times \R)$ such that
\begin{equation} \label{eq:MF}
 \int_{Q \times \R} \psi(\cdot,\eta)u \; d\nu(\eta,u) = F,
\end{equation}
i.e., of a Fredholm integral equation on measures.  
Approximate solution  of related problems by associated  mean-field Langevin dynamics
has been discussed in~\cite{Nitanda2022mfld,mei2018mean,rotskoff2018neural,Hu2021mfld,chizat2018global,Chizat2022mfld}.
with entropy regularisers added in \cite{Nitanda2022mfld,mei2018mean,Hu2021mfld,Chizat2022mfld}.
\end{remark}
}
%
%
\section{Ritz-Galerkin discretization and regularization} \label{sec:R-GT}
\subsection{Variational formulation}
Let $L \subset L^2(Q)$ denote a closed subspace.
Then  the corresponding Ritz-Galerkin discretization
of the Fredholm Training Problems \ref{prob:FFUNC} and \ref{prob:PPOINT}  reads as follows.
\begin{problem} \label{prob:RGVAR}
Find  $u_L \in L$ such that
\begin{equation} \label{eq:GALERKIN}
 (\cG u_L, \cG v)_\pi = (F,\cG v)_\pi \qquad \forall v \in L,
\end{equation}
with $\pi$ still denoting the Lebesgue or  point measure.
\end{problem}
Again, existence of a best-approximate solution $u_L$  follows from Theorem~2.5 in \cite{Engl},
and asymptotic ill-conditioning is expected to be inherited from the continuous case.

\subsection{Algebraic formulation}
Let the ansatz space $L = L_n$ have finite dimension $n$ and let $\varphi_1, \dots, \varphi_n$ be a basis of $L_n$.
Then any $v \in L_n$ can be identified with the coefficient vector $V=(V_i)_{i=1}^n\in \R^n$ of its unique basis representation
\begin{equation} \label{eq:basis}
v= \sum_{i=1}^n V_i \varphi_i \in L_n.
\end{equation}
\begin{example} \label{ex:1}
Let $Q = \bigcup_{i=1}^n q_i$ be a decomposition of $Q$ into disjoint  
 {hypercubes} $q_i$ with maximal  {diameter} $h$. Then 
\[
\varphi_i(\eta) = \left\{ 
\begin{array}{cc}
|q_i|^{-1/2}&\text{if } \eta \in q_i\\
0& \text{otherwise}
\end{array}
\right .
\]
is an orthonormal basis of the subspace $L_n \subset L^2(Q)$ of piecewise constant functions on $q_i$,
and the coefficients $V_i$ of  $v\in L_n$  agree with its scaled values on $q_i$, $i=1,\dots,n$.
\end{example}
Introducing the semi-discrete, linear operators
\begin{equation}
\begin{array}{lclrll}
G &=& (\cG \varphi_1, \dots, \cG \varphi_n) & \in  L^2(X,\pi) \times \R^n   \colon &\R^n \to L^2(X,\pi),\\
G^*  &=& ((\cG \varphi_i, \cdot)_\pi)_{i=1}^n  & \in \R^n \times L^2(X,\pi) \colon & L^2(X,\pi) \to \R^n
\end{array}
\end{equation}
we have $\cG u_L = G U$ with $u_L = \sum_{i=1}^n U_i \varphi_i$ and  $U =(U_i)_{i=1}^n \in \R^n$.
Utilizing
\begin{equation} 
G^*G = ((\cG \varphi_i,\cG \varphi_j)_\pi)_{i,j=1}^n \in \R^{n,n}
\end{equation}
the variational equality \ref{eq:GALERKIN} can be equivalently rewritten as the linear system
\begin{equation} \label{eq:ALGPROBSD}
G^* G U = G^* F
\end{equation}
for the unknown coefficient vector $U$ of $u_L$.
Observe that   $G^*G = A^\top A$ holds with
\begin{equation} \label{eq:SM}
A = (a_{ij}) \in \R^{M\times n}, \;\; a_{ij} =  (\cG \varphi_j)(x_i), \;\;  i=1,\dots M, \; j=1,\dots n,\;
\end{equation}
in the case (\ref{eq:DATA}b) of given point data.

\subsection{Tikhonov regularization} \label{subsec:TIKHONOV}
We only consider the case (\ref{eq:DATA}b) of given point data.
In order to ensure uniqueness of approximate solutions, we select some $\varepsilon > 0$
and introduce the following Tikhonov regularization of the discretized Problem~ \ref{prob:RGVAR}.
\begin{problem} \label{prob:ALGPROB}
Find a solution $U_\varepsilon \in \R^n$ of the regularized linear system
\begin{equation} \label{eq:ALGPROB}
(\varepsilon I + A^\top A) U_{\varepsilon} = A^\top b
\end{equation}
with identity matrix $I \in \R^{n\times n}$, $b = (F(x_i))_{i=1}^M \in \R^M$, and $\varepsilon >0$.
\end{problem}

Problem \ref{prob:ALGPROB} is equivalent to minimizing the regularized discrete loss functional
\begin{equation} \label{eq:REGLI}
{\mathbb J}_{\varepsilon}(V)= \lVert AV - b \rVert^2_{2,\R^M} + \varepsilon \lVert V \rVert^2_{2,\R^n}, \quad V \in \R^{n}.
\end{equation}
with $\lVert\,\cdot\,\rVert_{2,\R^M}$ and $\lVert\,\cdot\,\rVert_{2,\R^n}$ denoting the Euclidean norm in $\R^M$ and $\R^n$, respectively.
Existence and uniqueness of a solution is an immediate consequence of the Lax-Milgram lemma.

Moreover, Problem \ref{prob:ALGPROB} is also equivalent  with the algebraic formulation 
of the Ritz-Galerkin discretization of the regularized continuous problem
to find $u_\varepsilon \in L^2(Q)$ such that
\begin{equation} \label{eq:TICONT}
J_\varepsilon (u_\varepsilon) \leq J_\varepsilon (v)\qquad \forall v \in L^2 (Q)
\end{equation}
with regularized loss functional
\[
J_{\varepsilon}(v) =  \varepsilon (v,v) + \cJb(v) =   \varepsilon (v,v)  +  \frac{1}{M} \sum_{j=1}^M (F(x_j) - \cG v(x_j))^2
\]
provided that $\varphi_1, \dots, \varphi_n$ is an orthonormal basis of $L_n$ (see, e.g., Example~\ref{ex:1}).
Indeed, we have  ${\mathbb J}_{\varepsilon}(V) = J_{\varepsilon}(v)$ for $v= \sum_{i=1}^n V_i \varphi_i $ in this case.

 {
We conclude this section by a residual error estimate  for the loss functional,
as obtained for the ansatz space $L_n$ of piecewise constant functions introduced in Example~\ref{ex:1}. 
Let $U_{\varepsilon, h}\in \R^n$ denote the  solution of the corresponding Tikhonov regularized Problem~\ref{prob:ALGPROB}, and  set
\[
u_{\varepsilon, h} = \sum_{i=1}^n \left(U_{\varepsilon, h}\right)_i \varphi_i \in L_n \subset L^2(Q).
\]
\begin{proposition} \label{prop:RESERR}
Assume that the kernel function $\psi$ satisfies \eqref{eq:BNDDI} and $N(\cG_M^*)= \{0\}$. 
Assume further  that the solution $u \in L^2(Q)$ of Problem~\ref{prob:PPOINT} satisfies the smoothness condition $u \in R(\cG_M^* \cG_M)$ 
and that the solution $u_{\varepsilon}$ of the regularized problem \eqref{eq:TICONT} satisfies the regularity condition $u_{\varepsilon} \in H^1(Q)$. 
Then the  a priori residual estimate
\begin{equation}
0 \leq J(u_{\varepsilon, h})\leq \displaystyle C \left(\varepsilon^2 + \varepsilon^{-2}h^2\right)
\end{equation}
for the continuous loss functional $J$ introduced in \eqref{eq:COLO} holds with a constant $C$  depending only on $u$ and $u_{\varepsilon}$.
\end{proposition}
\begin{proof}
The minimization problem \eqref{eq:TICONT}  can be equivalently written in variational form
\[
a(u_{\varepsilon},v) = \ell(v) \quad \forall v \in L^2(Q)
\]
with bilinear form $a(\cdot, \cdot)$ and linear functional $\ell$ given by
\[
a(v,w)=\varepsilon (v,w) + \frac{1}{M}\sum_{i=1}^M \cG_M v(x_i) \cG_M w(x_i), \qquad \ell(v) =  \frac{1}{M}\sum_{i=1}^M F (x_i) \cG_M v(x_i).
\]
We have $\ell \in L^2(Q)'$ and  boundedness 
\[
|a(v,w)| \leq c \|v\|\|w\| \quad \forall v,  w \in L^2(Q)
\]
with a constant $c \geq \varepsilon + \|\cG_M\|^2$ independent of $M$ follows from uniform boundedness of $\cG_M$ (see Subsection~\ref{subsec:PD}).
Hence, C\'ea's lemma and  well-known results on  Poincar\'e inequalities on convex domains, cf., e.g.~\cite{bebendorf2003note}[Theorem 3.2]  imply
the discretization error bound 
\begin{equation} \label{eq:CEA}
\|u_{\varepsilon} - u_{\varepsilon, h}\| \leq c\; \varepsilon^{-1} \inf_{v \in L_n}\|u_{\varepsilon}- v\| \leq c\pi^{-1}\; \varepsilon^{-1} h \|u_{\varepsilon}\|_{H^1(Q)}
\end{equation}
The regularization error bound
\begin{equation} \label{eq:REGULERR}
\|u- u_{\varepsilon}\| \leq c\;  \varepsilon
\end{equation}
with another constant $c$ depending only on $u$ is a consequence of \cite{Engl}[Theorem 4.3].
The desired residual error estimate 
\[
J(u_{\varepsilon, h})\leq \left( \|\cG_M\| \|u- u_{\varepsilon}\| +  \|\cG_M\| \|u_{\varepsilon} - u_{\varepsilon,h}\|  \right)^2
\leq  c \left( \varepsilon +  \varepsilon^{-1}h \right)^2 \leq 2 c \left(  \varepsilon^2 +   \varepsilon^{-2}h^2 \right)
\]
with another constant $c$  depending only on $u$ and $u_{\varepsilon}$
now follows from $J(u)=0$, the uniform boundedness of $\cG_M$, and the error estimates  \eqref{eq:CEA} and \eqref{eq:REGULERR}.
\end{proof}
}


%
%
\section{Algebraic solution by tensor trains} \label{sec:TT}

In this section we consider the algebraic solution of the regularized linear system \eqref{eq:ALGPROB} in Problem~\ref{prob:ALGPROB}
for high dimension $d$ and resulting large number of unknowns $n$ by tensor train methods.

\subsection{Tensor formats}

Tensors are multilinear mappings that can be represented as functions $T \colon \{1,\dots, D_1\} \times \dots \times \{1, \dots, D_\ord\} \to \Cc$
or, equivalently,  as multidimensional arrays $\Tt \in \Cc^{D_1 \times \dots \times D_\ord}$ with dimensions or modes $D_k \in \N$, $k=1,\dots, \ord$, and order $\ord$ of $\Tt$.
We will write $\Tt \in \R^{D_1 \times \dots \times D_\ord}$, if all components $\Tt_{i_1, \dots, i_\ord}$,  $i_k = 1, \dots, D_k$, $k=1,\dots, \ord$, are real numbers.
Throughout the following, tensors are denoted by bold letters.
When certain indices of a tensor are fixed, we use colon notation (cf.~Python or Matlab) to indicate the remaining modes, 
e.g., $\Tt_{:,i_2, \dots, i_{\ord-1},:} \in \Cc^{D_1 \times D_\ord}$.

The sum of tensors and the product with scalars is defined componentwise.
Extending  usual matrix multiplication, the product or index contraction $\Tt \cdot \Ut \in \Cc^{D_1 \times \dots \times D_\ord \times D_1'' \times \dots \times D_{\ord''}'' }$ of two tensors
$\Tt \in \Cc^{D_1 \times \dots \times D_\ord \times D_1' \times \dots \times D_{\ord '}'}$ and $\Ut \in \Cc^{D_1' \times \dots \times D_{\ord'}' \times D_1'' \times \dots \times D_{\ord''}'' }$   is defined by 
\[
(\Tt  \cdot \Ut)_{i_1, \dots, i_\ord, j_1, \dots, j_{\ord''}} = \sum_{k_1=1}^{D_1'} \dots \sum_{k_{\ord '}=1}^{D_{\ord'}'} \Tt_{i_1, \dots,i_\ord, k_1, \dots, k_{\ord'}} \Ut_{k_1, \dots, k_{\ord'},  j_1, \dots, j_{\ord''}}.
\]
Note that the sum of the tensors $\Tt_{\ell} \in \Cc^{D_1 \times \cdots \times D_{\ord}}$, $\ell=1,\dots,r$, can be regarded as an index contraction of 
$ \Tt \in \Cc^{r \times D_1 \times \cdots \times D_\ord}$ with $\Tt_{\ell, i_1,\dots, i_\ord}= (\Tt_{\ell})_{i_1,\dots, i_\ord}$ 
and ${\mathbf 1} \in \R^{r}$ with ${\mathbf 1}_{\ell} =1$, $\ell=1, \dots, r$.
We will also make use of the tensor product $\Tt \otimes \Ut \in \Cc^{D_1 \times \dots \times D_\ord \times D_1' \times \dots \times D_{\ord'}' }$
of $\Tt \in \Cc^{D_1 \times \dots \times D_\ord}$ and $\Ut \in \Cc^{D_1' \times \dots \times D_{\ord'}' }$ which is defined by
\[
(\Tt \otimes \Ut)_{i_1, \dots,i_\ord, j_1, \dots, j_{\ord'}} = \Tt_{i_1, \dots,i_\ord} \Ut_{j_1, \dots, j_{\ord'}}
\]
and coincides with the dyadic product  $\Tt \Ut^\top \in \Cc^{D_1\times D_1'}$ of two vectors $\Tt \in \Cc^{D_1}$ and $\Ut \in \Cc^{D_1'}$ for $\ord = \ord' =1$.

As the number of elements of a tensor  grows exponentially with order $\ord$, the storage of high-dimensional tensors may become infeasible, if $\ord$ becomes large.
This motivates suitable formats for tensor representation and approximation.
Obviously, the storage of a \emph{rank-one tensor}  $\Tt\in \Cc^{D_1 \times \dots \times D_\ord}$ that  can be written as the tensor product 
\begin{equation} \label{eq:R1}
	\Tt = \Tt^{(1)} \otimes \dots \otimes \Tt^{(\ord)}
\end{equation}
of $\ord$ vectors $\Tt^{(k)} \in \Cc^{D_k}$, $k=1, \dots , \ord$, only grows linearly with $\ord D_{\text{max}}$ and  $D_{\text{max}} = \max \{D_1, $ $ \dots, D_\ord\}$.
As a generalization of \eqref{eq:R1}, a tensor   $\Tt\in \Cc^{D_1 \times \dots \times D_\ord}$ is said to be in  the \emph{canonical format}~\cite{hitchcock1927expression}, 
if it can be expressed as a finite sum of rank-one tensors according to
\begin{equation} \label{eq:CT}
	\Tt = \sum_{\ell=1}^r \Tt^{(1)}_{\ell,:} \otimes \dots \otimes \Tt^{(\ord)}_{\ell,:}.
\end{equation}
Here,  the tensors $\Tt^{(k)} \in \Cc^{r \times D_k}$, $k=1, \dots, \ord$, are called cores and
the number $r$ is called the canonical rank of $\Tt$.
Note that the required storage grows with order $\cO(r \ord D_{\text{max}})$.

The \emph{tensor train} (TT) format~\cite{oseledets2009breaking, OSELEDETS2011} is a further generalization of \eqref{eq:CT}.
Here, the basic idea  is to decompose a tensor into a chain-like network of low-dimensional tensors coupled by so-called TT ranks $r_0, \dots r_{\ord}$ according to
\begin{equation}\label{eq: TT}
	\Tt = \sum_{\ell_0=1}^{r_0} \dots \sum_{\ell_\ord=1}^{r_\ord}  \Tt^{(1)}_{\ell_0,:, \ell_1} \otimes \dots \otimes \Tt^{(\ord)}_{\ell_{\ord-1},:, \ell_\ord},
\end{equation}
with  TT cores $\Tt^{(k)} \in \Cc^{r_{k-1} \times D_k \times r_k}$, $k=1,\dots, \ord$, and $r_0 = r_\ord = 1$.
With maximal TT rank  $r_{\text{max}}= \max \{r_0,  \dots, r_\ord\}$,
the storage consumption for tensors in the TT format grows like $\cO(r^2_{\text{max}} \ord D_{\text{max}})$
and thus is still linearly bounded in terms of the maximal dimension  $D_{\text{max}}$ and the order $ \ord$ of $\Tt$.
The ranks $r_0,  \dots, r_\ord$ not only determine the required storage, 
but also have a strong influence on the quality of best approximation of a given tensor in full format by a tensor train.
Note that such a best approximation in TT format with bounded ranks always exists~\cite{falco2012minimal} 
and is even exact for suitably chosen ranks~\cite[Theorem 1]{holtz2012manifolds}.
For further information about tensor formats, we refer, e.g., to~\cite{kolda2009tensor, hackbusch2012tensor}.

In order to provide a factorization of the integral operator $\cG$ by factorization of the kernel $\psi$, we will make use of functional tensor networks with discrete as well as continuous modes. 
More precisely, a functional tensor with one discrete and one continuous mode is a function $T \colon \{1, \dots, D\} \times Q \to \Cc$, where $Q\subseteq \R$ 
denotes the domain of the continuous mode.
We write $\Tt_{i, \eta} = T(i, \eta)$, $i=1, \dots, D$, $\eta \in Q$, for the components of the corresponding generalized array $\Tt \in \Cc^{D\times Q}$.
Replacing summation by integration, index contraction $\Tt \cdot \Ut \in \Cc^{D\times D'}$ of $\Tt$ with the functional tensor $\Ut\in \Cc^{Q\times D'}$ is defined by
\[
(\Tt \cdot \Ut)_{i,j} = \int_Q \Tt_{i,\eta}\Ut_{\eta ,j}\; d\eta.
\]
Functional tensor trains  (FTTs)~\cite{GORODETSKY2019}  provide a decomposition of 
$\Tt \in \Cc^{Q_1 \times \cdots \times Q_\ord}$  into  cores $\Tt^{(k)}\in \Cc^{r_{k-1}\times Q_k \times r_k}$ in analogy to \eqref{eq: TT}.

Figure~\ref{fig: formats} shows the diagrammatic notation of tensors in rank-one, canonical, TT, and FTT format. 
We depict vectors, matrices, and tensors as circles with different arms indicating the set of modes.
Index contraction of two or more tensors is represented by connecting corresponding arms.

\def\g{0.9}
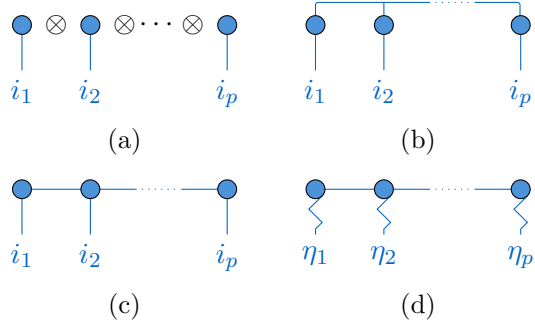
\begin{figure}[htb]
	\centering
	\begin{subfigure}[b]{0.24\textwidth}
		\centering
		\begin{tikzpicture}
			\node[] at (0,-1) {};
			\node[] at (0.5*\g,0) {$\otimes$};
			\node[] at (1.5*\g,0) {$\otimes$};
			\node[] at (2*\g,0) {$\cdots$};
			\node[] at (2.5*\g,0) {$\otimes$};
			\draw[black] (0,0) --++ (0,-0.6) node [label={[shift={(0,-0.65)},text depth=0]$i_1$}] {};
			\draw[black] (\g,0) --++ (0,-0.6) node [label={[shift={(0,-0.65)},text depth=0]$i_2$}] {};
			\draw[black] (3*\g,0) --++ (0,-0.6) node [label={[shift={(0,-0.65)},text depth=0]$i_\ord$}] {};
			\node[draw,shape=circle,fill=black!70, inner sep=2.5pt] at (0,0){};
			\node[draw,shape=circle,fill=black!70, inner sep=2.5pt] at (\g,0){};
			\node[draw,shape=circle,fill=black!70, inner sep=2.5pt] at (3*\g,0){};
		\end{tikzpicture}
		\caption*{(a)}
	\end{subfigure}
	\begin{subfigure}[b]{0.24\textwidth}
		\centering
		\begin{tikzpicture}
			\node[] at (0,-1) {};
			\draw[black, rounded corners=1pt] (0,0) --++ (0,0.3) --++ (1.66*\g,0);
			\draw[black, dotted] (1.66*\g,0.3) -- ++ (0.66*\g,0) ;
			\draw[black, rounded corners=1pt] (2.33*\g,0.3) --++ (0.66*\g,0) --++ (0,-0.3);
			\draw[black] (\g,0) --++ (0,0.3) ;
			\draw[black] (0,0) --++ (0,-0.6) node [label={[shift={(0,-0.65)},text depth=0]$i_1$}] {};
			\draw[black] (\g,0) --++ (0,-0.6) node [label={[shift={(0,-0.65)},text depth=0]$i_2$}] {};
			\draw[black] (3*\g,0) --++ (0,-0.6) node [label={[shift={(0,-0.65)},text depth=0]$i_\ord$}] {};
			\node[draw,shape=circle,fill=black!70, inner sep=2.5pt] at (0,0){};
			\node[draw,shape=circle,fill=black!70, inner sep=2.5pt] at (\g,0){};
			\node[draw,shape=circle,fill=black!70, inner sep=2.5pt] at (3*\g,0){};
		\end{tikzpicture}
		\caption*{(b)}
	\end{subfigure}\\[0.25cm]
	\begin{subfigure}[b]{0.24\textwidth}
	\centering
		\begin{tikzpicture}
			\node[] at (0,-1) {};
			\draw[black] (0,0) -- ++ (\g,0) ;
			\draw[black] (\g,0) -- ++ (0.66*\g,0) ;
			\draw[black, dotted] (1.66*\g,0) -- ++ (0.66*\g,0) ;
			\draw[black] (2.33*\g,0) -- ++ (0.66*\g,0) ;
			\draw[black] (0,0) --++ (0,-0.6) node [label={[shift={(0,-0.65)},text depth=0]$i_1$}] {};
			\draw[black] (\g,0) --++ (0,-0.6) node [label={[shift={(0,-0.65)},text depth=0]$i_2$}] {};
			\draw[black] (3*\g,0) --++ (0,-0.6) node [label={[shift={(0,-0.65)},text depth=0]$i_\ord$}] {};
			\node[draw,shape=circle,fill=black!70, inner sep=2.5pt] at (0,0){};
			\node[draw,shape=circle,fill=black!70, inner sep=2.5pt] at (\g,0){};
			\node[draw,shape=circle,fill=black!70, inner sep=2.5pt] at (3*\g,0){};
		\end{tikzpicture}
	\caption*{(c)}
	\end{subfigure}
	\begin{subfigure}[b]{0.24\textwidth}
		\centering
		\begin{tikzpicture}
			\node[] at (0,-1) {};
			\draw[black] (0,0) -- ++ (\g,0) ;
			\draw[black] (\g,0) -- ++ (0.66*\g,0) ;
			\draw[black, dotted] (1.66*\g,0) -- ++ (0.66*\g,0) ;
			\draw[black] (2.33*\g,0) -- ++ (0.66*\g,0) ;
			\draw[black,decoration={zigzag}, decorate] (0,0) --++ (0,-0.6) node [label={[shift={(0,-0.6)},text depth=0]$\eta_1$}] {};
			\draw[black,decoration={zigzag}, decorate] (\g,0) --++ (0,-0.6) node [label={[shift={(0,-0.6)},text depth=0]$\eta_2$}] {};
			\draw[black,decoration={zigzag}, decorate] (3*\g,0) --++ (0,-0.6) node [label={[shift={(0,-0.6)},text depth=0]$\eta_\ord$}] {};
			\node[draw,shape=circle,fill=black!70, inner sep=2.5pt] at (0,0){};
			\node[draw,shape=circle,fill=black!70, inner sep=2.5pt] at (\g,0){};
			\node[draw,shape=circle,fill=black!70, inner sep=2.5pt] at (3*\g,0){};
		\end{tikzpicture}
		\caption*{(d)}
	\end{subfigure}
	\caption{Graphical notation of tensor formats: 
	(a) A tensor in rank-one format, given by the tensor product of $\ord$ vectors. 
	(b) A tensor in the canonical format, given by the contraction of $\ord$ matrices on the common rank index. 
	(c)~A tensor in TT format, given by a network of pairwise coupled tensors. 
	Here, the first and the last TT core are regarded as matrices, because $r_0 = r_\ord = 1$. 
	(d) A tensor in FTT format, given by a TT-like network of tensors with one continuous mode. Discrete modes are represented by straight lines and continuous modes by zigzag lines.}
	\label{fig: formats}
\end{figure}


\subsection{Tensor decomposition of basis and stiffness matrix}\label{sec: tensor factorization}
As a first step towards a tensor train formulation of the regularized linear system \eqref{eq:ALGPROB}, we assume that the basis functions $\varphi_1, \dots, \varphi_n$ of the ansatz space $L_n$ utilized in \eqref{eq:basis}  are separable in the sense that
\begin{equation} \label{eq:PROD}
	\varphi_j(\eta) = \varphi^{(1)}_{j_1} (\eta_1) \cdot \ldots \cdot \varphi^{(s)}_{j_s} (\eta_s), \quad \eta=(\eta_1,\dots, \eta_s) \in Q = Q_1 \times \cdots \times Q_s,
\end{equation}
holds with $1 \leq j_k \leq n_k$ and some enumeration
\begin{equation} \label{eq:ENUM}
j= j(j_1, \dots, j_s) = 1,\dots, n= n_1\cdot \ldots \cdot n_s, \qquad j_k = 1,\dots, n_k,\quad k=1,\dots,s.
\end{equation}
\begin{example}
Let $Q_k = \bigcup_{j=1}^{n_k} q_j^{(k)}$ be a decomposition of $Q_k \subset \R^{n_k}$ into disjoint intervals $q_j^{(k)}\subset \R$, $j = 1,\dots, n_k$.
Then the piecewise constant basis functions  $\varphi_1, \dots, \varphi_n$ introduced in Example~\ref{ex:1} 
can be written as a product of the form \eqref{eq:PROD} with
\[
\varphi_{j}^{(k)}(\eta_k)= \left\{ 
\begin{array}{cc}
|q_j^{(k)}|^{-1/2}&\text{if } \eta_k \in q_j^{(k)}\\
0& \text{otherwise}
\end{array}
\right .
.
\]
\end{example}

By definition of the tensor product, we have the following proposition.
\begin{proposition} \label{prop:BASISPROD}
Assume that the basis functions $\varphi_1, \dots, \varphi_n$ of the ansatz space $L_n$ are separable in the sense of \eqref{eq:PROD}.
Then the functional tensor 
\[
{\mathbf \Phi} \in \R^{n_1 \times \cdots \times n_s \times Q_1 \times \cdots \times Q_s}, \qquad
\mathbf{\Phi}_{j_1, \dots, j_s, \eta_1 , \dots, \eta_s } =  \varphi^{(1)}_{j_1} (\eta_1) \cdot \ldots \cdot \varphi^{(s)}_{j_s} (\eta_s),
\]
is a  rank-one tensor with the decomposition
\begin{equation} \label{eq:BASDEC}
\mathbf{\Phi} = \mathbf{\Phi}^{(1)} \otimes \cdots \otimes  \mathbf{\Phi}^{(s)} 
\end{equation}
into  $\mathbf{\Phi}^{(k)} \in \R^{n_k \times Q_k}$ defined by  $\mathbf{\Phi}^{(k)}_{j_k, \eta_k}= \varphi_{j_k}^{(k)}(\eta_k)$, $j_k = 1,\dots, n_k$, for $k=1,\dots, s$.
\end{proposition}

In the next step, we derive a decomposition of the functional tensor  $\mathbf{\Psi} \in \R^{M \times Q_1 \times \cdots \times Q_s}$ with components
\begin{equation} \label{eq:FUNTENS}
	\mathbf{\Psi}_{i, \eta_1, \dots, \eta_s} = \psi(x_i, \eta),  \qquad  i =1, \dots, M, \quad \eta = (\eta_1, \dots, \eta_s)^\top \in Q= Q_1 \times \cdots \times Q_s,
\end{equation} 
associated with the actual kernel function $\psi$ and the given data points $x_i\in X$, $i=1, \dots, M$.
For ease of presentation,  we focus on kernels $\psi \colon X \times Q \to \R$ with $X, \;Q \subset \R^s$.
Obviously,  radial kernel functions  \eqref{eq:RADIAL} have this property with  $s=d$,
and in the case of ridge kernel functions   \eqref{eq:RIGDE} we can identify $X$ with its  embedding $\{(x,1)^\top\;|\; x \in X\}\subset \R^s$, $s=d+1$.

We further assume that the actual kernel function $\psi$ can be expressed (or approximated) 
as a sum of products of univariate functions according to
\begin{equation} \label{eq:FACKER}
	\psi(x, \eta) = \sum_{\ell=1}^r \psi^{(1)}_{\ell} (x_1, \eta_1) \cdot \ldots \cdot \psi^{(s)}_{\ell} (x_s, \eta_s).
\end{equation}
\begin{example} \label{ex:FACKER}
The ridge kernel function $\psi(x,\eta)= \sigma (x \cdot \eta)$
with trivial activation function  $\sigma = \text{Id}$ can be written in the form \eqref{eq:FACKER}
with $r=s$ and $\psi^{(k)}_{\ell} (x_k, \eta_k) = x_k \eta_k$ for $k=\ell$ and  $\psi^{(k)}_{\ell} (x_k, \eta_k) = 1$ otherwise.
\end{example}
Explicit constructions of approximate factorizations of ridge kernels with more relevant activation functions can be found in Appendix~\ref{app}.
Note that even in the case of real-valued kernel functions  $\psi$, the factors  $\psi^{(k)}_{\ell}$ might have complex values (see, e.g. \ref{A:HEAVI}).
\begin{example} \label{ex:RADIAL}
The radial kernel function with Gaussian activation function given in \eqref{eq:GAUSSAC}
takes the form \eqref{eq:FACKER} with $r=1$ and 
\[
\psi^{(k)}_{1} (x_k, \eta_k) = \exp\left(-\frac{ (x_k - \eta_k)^2}{2\kappa^2}\right),\quad k=1,\dots,d.
\]
\end{example}

Utilizing the given  data points $x_i =(x_{i,k})_{k=1}^s \in X$, $i=1, \dots,M$,
we now introduce the auxiliary functional tensor $\hat{\mathbf{\Psi}} \in \Cc^{M^s \times Q_1 \times \cdots \times Q_s}$ with  canonical decomposition 
\begin{equation} \label{eq:HELP}
	\hat{\mathbf{\Psi}} = \sum_{\ell=1}^r \hat{\mathbf{\Psi}}^{(1)}_{\ell, :, :} \otimes \dots \otimes \hat{\mathbf{\Psi}}^{(s)}_{\ell, :, :}
\end{equation}
into cores $\hat{\mathbf{\Psi}}^{(k)} \in \Cc^{r \times M \times Q_k}$ with components 
$\hat{\mathbf{\Psi}}^{(k)}_{\ell, i, \eta_k} = \psi^{(k)}_\ell (x_{i,k}, \eta_k)$.
\begin{proposition} \label{prop:KERNEL}
Assume that the  kernel function  $\psi$ allows for the representation \eqref{eq:FACKER}.
 Then  the  functional tensor $\mathbf{\Psi} \in \Cc^{M \times Q_1 \times \cdots \times Q_s}$ can be decomposed according  to
 \begin{equation} \label{eq:KERDEC}
 \mathbf{\Psi} = \Delta \cdot \hat{\mathbf{\Psi}}
 \end{equation}
 with the delta tensor $\Delta \in \R^{M \times M^{s}}$ defined by 
\begin{equation} \label{eq:DELTA}
\Delta_{i, j_1, \dots, j_s} =   \delta_{i,j_1} \cdot \ldots \cdot \delta_{i,j_s}
\end{equation}
and $\delta_{i,j_k}$ denoting  the Kronecker delta.
\end{proposition}
\begin{proof} The desired identity \eqref{eq:KERDEC} follows from
\begin{align*}
	(\Delta \cdot \hat{\mathbf{\Psi}})_{i, \eta_1, \dots, \eta_s}  &= \sum_{j_1=1}^M \dots \sum_{j_s=1}^M \Delta_{i, j_1, \dots , j_s} \hat{\mathbf{\Psi}}_{j_1, \dots, j_s,  \eta_1, \dots,\eta_s} \\
	&= \sum_{\ell=1}^r \left( \sum_{j_1=1}^M \delta_{i,j_1} \hat{\mathbf{\Psi}}^{(1)}_{\ell, j_1, \eta_1} \right) \cdots \left( \sum_{j_s=1}^M \delta_{i,j_s} \hat{\mathbf{\Psi}}^{(s)}_{\ell, j_s, \eta_s} \right) \\
	&= \sum_{\ell=1}^r \hat{\mathbf{\Psi}}^{(1)}_{\ell, i, \eta_1} \cdots \hat{\mathbf{\Psi}}^{(s)}_{\ell, i, \eta_s} 
	= \sum_{\ell=1}^r \psi^{(1)}_{\ell} (x_{i,1}, \eta_1) \cdots \psi^{(s)}_{\ell} (x_{i,s}, \eta_s) = \mathbf{\Psi}_{i, \eta_1, \dots, \eta_s}.\qedhere
\end{align*}

\end{proof}
The construction of  the functional tensor $\mathbf{\Psi}$ is illustrated in Figure~\ref{fig: kernel decomposition}~(a).
Note that  the auxiliary tensor $\Delta$ does not explicitly appear in our numerical computations.
We will also exploit the FTT format to represent the auxiliary tensor $\hat{\Psit}$, if it allows for a low-rank decomposition.
In particular, monomials of the form $(x \cdot \eta)^{p}$ can be expressed in a compact way using a tensor train-like coupling as we show in Appendix~\ref{app}.
Figure~\ref{fig: kernel decomposition}~(b) shows the corresponding FTT network for $\Psit$.

Now we are in the position to provide a representation of the stiffness matrix $A=(a_{ij}) \in \R^{M \times n}$ defined in \eqref{eq:SM} as a tensor network.

\begin{figure}[htb]
	\centering
	\begin{subfigure}[b]{0.49\textwidth}
		\centering
		\begin{tikzpicture}
			\def\s{0.35}
			\draw[black, rounded corners = 1pt] (0,0) --++ (\s,\s) --++ (3.5,0);
			\draw[black] (2,0) --++ (\s,\s);
			\draw[black, dotted] (3.5+\s,\s) --++ (1,0);
			\draw[black, rounded corners = 1pt] (4.5+\s,\s) --++ (1.5,0) --++ (-\s,-\s);
			
			\draw[black,decoration={zigzag}, decorate] (0,0) --++ (0,0.6) node [label={[text depth=0]$\eta_1$}] {};
			\draw[white,decoration={zigzag}, decorate, line width=3pt] (2,0) --++ (0,0.6) ;
			\draw[white,decoration={zigzag}, decorate, line width=3pt] (6,0) --++ (0,0.6) ;
			\draw[black,decoration={zigzag}, decorate] (2,0) --++ (0,0.6) node [label={[text depth=0]$\eta_2$}] {};
			\draw[black,decoration={zigzag}, decorate] (6,0) --++ (0,0.6) node [label={[text depth=0]$\eta_s$}] {};
			
			\draw[black] (0,0) --++ (0,-0.6) -- (3,-1.2);
			\draw[black] (2,0) --++ (0,-0.6) -- (3,-1.2);;
			\draw[black] (6,0) --++ (0,-0.6) -- (3,-1.2);;
			
			\draw[black] (3,-1.2) --++ (0,-0.6) node [label={[shift={(0,-0.6)},text depth=0]$i$}] {};
			
			\node[draw,shape=circle,fill=black!70, inner sep=2.5pt] at (0,0){};
			\node[draw,shape=circle,fill=black!70, inner sep=2.5pt] at (2,0){};
			\node[draw,shape=circle,fill=black!70, inner sep=2.5pt] at (6,0){};
			
			\node[draw,shape=circle,fill=Green!70, inner sep=2.5pt] at (3,-1.2){};
			
		\end{tikzpicture}
	\caption*{(a)}
	\end{subfigure}
	\hfill
	\begin{subfigure}[b]{0.49\textwidth}
		\centering
		\begin{tikzpicture}
			\def\s{0.35}
			
			\draw[black] (0,0) --++ (3.5,0);
			\draw[black, dotted] (3.5,0) --++ (1,0);
			\draw[black] (4.5,0) --++ (1.5,0);
			
			\draw[black,decoration={zigzag}, decorate] (0,0) --++ (0,0.6) node [label={[text depth=0]$\eta_1$}] {};
			\draw[white,decoration={zigzag}, decorate, line width=3pt] (2,0) --++ (0,0.6) ;
			\draw[white,decoration={zigzag}, decorate, line width=3pt] (6,0) --++ (0,0.6) ;
			\draw[black,decoration={zigzag}, decorate] (2,0) --++ (0,0.6) node [label={[text depth=0]$\eta_2$}] {};
			\draw[black,decoration={zigzag}, decorate] (6,0) --++ (0,0.6) node [label={[text depth=0]$\eta_s$}] {};
			
			\draw[black] (0,0) --++ (0,-0.6) -- (3,-1.2);
			\draw[black] (2,0) --++ (0,-0.6) -- (3,-1.2);;
			\draw[black] (6,0) --++ (0,-0.6) -- (3,-1.2);;
			
			\draw[black] (3,-1.2) --++ (0,-0.6) node [label={[shift={(0,-0.6)},text depth=0]$i$}] {};
			
			\node[draw,shape=circle,fill=black!70, inner sep=2.5pt] at (0,0){};
			\node[draw,shape=circle,fill=black!70, inner sep=2.5pt] at (2,0){};
			\node[draw,shape=circle,fill=black!70, inner sep=2.5pt] at (6,0){};
			
			\node[draw,shape=circle,fill=Green!70, inner sep=2.5pt] at (3,-1.2){};
			
		\end{tikzpicture}
		\caption*{(b)}
	\end{subfigure}
	\caption{Graphical notation of the functional decomposition of $\mathbf{\Psi}$ corresponding to the kernel $\psi(x, \eta)$ and data points $(x_j)_{j=1}^M$: (a) Using the canonical format for $\hat{\Psit}$ with cores coupled by a common rank. (b) Using the FTT format with a chain-like coupling. black circles represent the cores of $\hat{\mathbf{\Psi}}$, each having one continuous mode represented by a zigzag line. Contraction with $\Delta$ (green circle) merges all free modes of $\hat{\mathbf{\Psi}}$ into one. }
	\label{fig: kernel decomposition}
\end{figure}
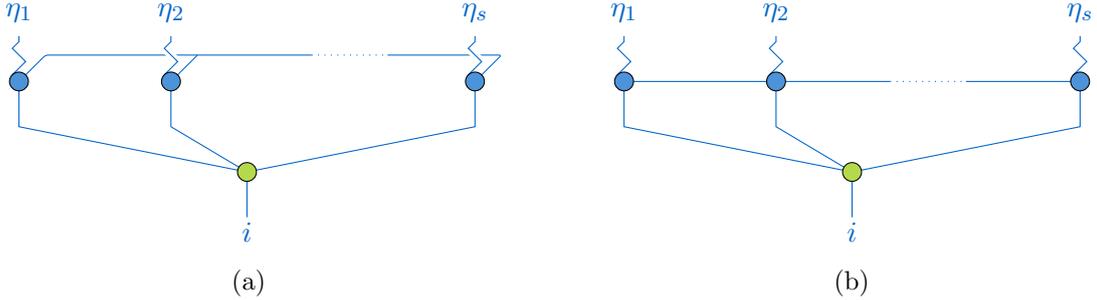


\begin{proposition} \label{prop:STIFFDEC}
Assume that the basis functions $\varphi_1, \dots, \varphi_n$ of the ansatz space $L_n$ are separable in the sense of \eqref{eq:PROD}
and that the  kernel function  $\psi$ allows for the representation~\eqref{eq:FACKER}.
Then the tensor 
\begin{equation} \label{eq:AREP}
 {\mathbf A} \in \R^{M \times n_1 \times \cdots \times n_s}, \quad  {\mathbf A}_{i,j_1,\dots, j_s} = a_{i,j(j_1,\dots,j_s)}
\end{equation}
with  enumeration $j(j_1,\dots,j_s)$ introduced in \eqref{eq:ENUM} allows for the factorization
\begin{equation} \label{eq:AFAC}
{\mathbf A}  =  \Delta \cdot \sum_{\ell = 1}^r {\mathbf A}_\ell^{(1)} \otimes \cdots \otimes {\mathbf A}_\ell^{(s)}, 
\quad {\mathbf A}_\ell^{(k)} = (\hat{\mathbf{\Psi}}^{(k)}_{\ell,:,:}\cdot {\mathbf \Phi}^{(k)}) \in \Cc^{M\times n_k}, \;\; k=1,\dots,s,
\end{equation}
with  ${\mathbf \Phi}^{(k)}$,  $\hat{\mathbf{\Psi}}^{(k)}$, and  $\Delta$ taken from   \eqref{eq:BASDEC},  \eqref{eq:HELP}, and \eqref{eq:DELTA}, respectively.
\end{proposition}
\begin{proof}
The desired identity follows from
\begin{align*}
\At_{i,j_1,\dots,j_s} =& \sum_{i_1}^M \dots \sum_{i_s}^M \Delta_{i,i_1, \dots,i_s}\sum_{\ell=1}^r \At_{\ell, i_1, j_1}^{(1)} \cdots \At_{\ell, i_s, j_s}^{(s)}\\
=& \sum_{\ell=1}^r \left( \sum_{i_1}^M \delta_{i,i_1}\At_{\ell, i_1, j_1}^{(1)} \right) \cdots \left( \sum_{i_s}^M\ \delta_{i,i_s}\At_{\ell, i_s, j_s}^{(s)} \right) 
=  \sum_{\ell =1}^r \At_{\ell, i, j_1}^{(1)} \cdots \At_{\ell, i, j_s}^{(s)}\\
=&  \sum_{\ell =1}^r \int_{Q_1} \psi_\ell(x_{i_1}, \eta_1)\varphi_{j_1}^{(1)}(\eta_1) \; d \eta_1 \dots \int_{Q_s} \psi_\ell(x_{i_s}, \eta_s)\varphi_{j_s}^{(s)}(\eta_s) \; d \eta_s\\
=& \int_Q  \left(\sum_{\ell =1}^r \psi_\ell(x_{i_1}, \eta_1)\cdots \psi_\ell(x_{i_s}, \eta_s) \right) \;\varphi_{j_1}^{(1)}(\eta_1)\cdots \varphi_{j_s}^{(s)}(\eta_s)\; d\eta\\
=& \int_Q  \psi(x_i, \eta) \varphi_{j(j_1,\dots, j_s)}(\eta) \; d\eta = a_{i, j(j_1, \dots, j_s)}.\qedhere
\end{align*}
\end{proof}
Utilizing the enumeration \eqref{eq:ENUM}, we introduce the tensor representation
\[
{\Ut} \in \R^{n_1 \times \cdots \times n_s}, \quad \Ut_{j_1, \dots, j_s}= U_{j(j_1,\dots, j_s)},
\]
of the unknown solution vector $U \in \R^n$ in \eqref{eq:ALGPROB}. Then, the regularized 
linear system \eqref{eq:ALGPROB} in Problem~\ref{prob:ALGPROB}  takes the form
\begin{equation} \label{eq:FULLPROB}
(\varepsilon {\mathbf I} + \At^\top\At) \Ut = \At^\top {\mathbf b}
\end{equation}
with  unit tensor ${\mathbf I}\in \R^{n_1 \times \cdots \times n_s}$ and ${\mathbf b}=b \in \R^M$.
Solving \eqref{eq:FULLPROB} is equivalent to minimizing the tensor analogue
\[
{\mathbf J}_{\varepsilon}(\Vt)
= \lVert \At\Vt - {\mathbf b} \rVert^2_{2,\R^M} + \varepsilon \lVert \Vt \rVert^2_{2,\R^{n_1 \times \cdots \times n_s}}, 
\quad \Vt \in  \R^{n_1 \times \cdots \times n_s},
\]
of the regularized discrete loss functional ${\mathbb J}_{\varepsilon}$ defined in \eqref{eq:REGLI}.
We now introduce the corresponding tensor-train approximation of the linear system \eqref{eq:FULLPROB}.
\begin{problem} \label{prob:TTAPROX}
Find  cores $\Ut^{(k)}\in \R^{r_{k-1}\times n_k \times r_k}$, $k=1,\dots,s$, with $r_0=r_s=1$ and prescribed ranks $r_k$, $k=1, \dots, s-1$
such that the resulting tensor train
\begin{equation} \label{eq:TTAPPROX}
\Ut_r =  \sum_{\ell_0=1}^{r_0} \dots \sum_{\ell_s=1}^{r_s}  \Ut^{(1)}_{\ell_0,:, \ell_1} \otimes \dots \otimes \Ut^{(s)}_{\ell_{s-1},:, \ell_s} 
\in \R^{n_1 \times \cdots \times n_s},
\end{equation}
is minimizing  the regularized discrete loss functional   ${\mathbf J}_{\varepsilon}$
over the submanifold of all tensor trains $\Vt$ of the form \eqref{eq:TTAPPROX}.
\end{problem}
Existence of a solution $\Ut_r$ of Problem~\ref{prob:TTAPROX} follows by compactness arguments.
Note that we have $\Ut_r = \Ut$, if the selection $r_1, \dots, r_{s-1}$ agrees with the minimal rank of  the unique solution $\Ut$ of \eqref{eq:FULLPROB},
because there exists an exact tensor train representation of $\Ut$ in this case (see \cite[Theorem 1]{holtz2012manifolds}).
As a consequence of non-uniqueness of  tensor train representations,  uniqueness of the cores  
$\Ut^{(k)} \in \R^{r_{k-1}\times n_k \times r_k}$ can not be expected.

\subsection{Algebraic solution by alternating linear systems}\label{sec: ARR}
The Tensor-Train Approximation Problem~\ref{prob:TTAPROX} allows for a variety of different methods for direct or iterative solution, such as pseudoinversion~\cite{KLUS2019, gelss2017tensor, GELSS2019, Batselier2021}
or core-wise linear updates~\cite{HOLTZ2012, KLUS2019, Dolgov2014,  Grasedyck2019}, respectively.
Here, we concentrate on iterative solution by \emph{alternating ridge regression} (ARR), cf.~\cite{KLUS2019},
because this approach significantly reduces computational cost and memory consumption in comparison with explicit computation of pseudoinverses
(see, e.g., \cite{KLUS2019} for details). 
These methods provide successive updates of the cores $\tilde{\Ut}^{(k)} \in \R^{r_{k-1} \times n_k \times r_k}$, $k=1,\dots, s$,
by solving a sequence of reduced linear problems.

More precisely, for given iterate $\tilde{\Ut}$ of the form \eqref{eq:TTAPPROX} and each fixed $k$, the cores $\tilde{\Ut}^{(1)}, \dots, \tilde{\Ut}^{(k-1)}$ are \emph{left-orthonormalized} and $\tilde{\Ut}^{(k+1)}, \dots, \tilde{\Ut}^{(s)}$ are \emph{right-orthonormalized} by successive singular value decompositions~\cite{oseledets2009breaking} such that the rows and columns, respectively, of certain reshapes of these cores form orthonormal sets.
Note that, without truncation, this procedure produces a different but equivalent representation of $\tilde{\Ut}$, say $\hat{\Ut}$.
The orthonormalized cores then define a retraction operator 
\[
\Qt_{\hat{\Ut},k}  \in \R^{(n_1\times \cdots \times n_s) \times (r_{k-1} \times n_k \times r_k)},
\]
which satisfies
\[
\Qt_{\hat{\Ut},k} \Vt^{(k)} =  \sum_{\ell_0=1}^{r_0} \dots \sum_{\ell_s=1}^{r_s} \bigotimes_{i=1}^{k-1}
\hat{ \Ut}^{(i)}_{\ell_{i-1},:, \ell_{i}} \otimes  \Vt^{(k)}_{\ell_{k-1},:,\ell_k} \otimes \bigotimes_{i=k+1}^{s}
\hat{ \Ut}^{(i)}_{\ell_{i-1},:, \ell_{i}} 
\]
for all $\Vt^{(k)}\in \R^{r_{k-1} \times n_k \times r_k}$ and is orthonormal in the sense that
\begin{equation} \label{eq:RETORTO}
	\Qt_{\hat{\Ut},k}^\top  \cdot \Qt_{\hat{\Ut}, k} = \It \in \R^{ (r_{k-1} \times n_k \times r_k) \times (r_{k-1} \times n_k \times r_k)}.
\end{equation}

Minimization with respect to the $k$-th core  then gives rise to the reduced  problem to
minimize
\[
\lVert \At \Qt_{\tilde{\Ut}, k}\Vt^{(k)} - {\mathbf b} \rVert^2_{2, \R^{M}}
+ \varepsilon \lVert \Qt_{\tilde{\Ut}, k} \Vt^{(k)} \rVert^2_{2,\R^{n_1\times \cdots \times n_s}}
\]
over all $\Vt^{(k)} \in \R^{r_{k-1} \times n_k \times r_k}$ or, equivalently,  to solve the reduced linear system
\begin{equation} \label{eq:REDSYY}
\Qt_{\tilde{\Ut}, k}^\top (\varepsilon \It + \At^\top \At) \Qt_{\tilde{\Ut}, k} \tilde{\Ut}^{(k)}_{\text{new}} 
= \Qt_{\tilde{\Ut}, k}^\top\At^\top \bt 
\end{equation}
for the new core $\tilde{\Ut}^{(k)}_{\text{new}}\in \R^{r_{k-1} \times n_k \times r_k}$
(see, e.g. \cite{HOLTZ2012, KLUS2019, Dolgov2014,  Grasedyck2019} for details).
The unique solution  $\tilde{\Ut}^{(k)}_{\text{new}}$ of \eqref{eq:REDSYY} 
can be computed, e.g., by Cholesky decomposition of the symmetric and positive definite coefficient matrix.

The corresponding update of cores is typically applied in alternating order, i.e., from $k=1$ to $k=s$ and then back to $k=1$.
As the same regularization parameter $\varepsilon>0$ is used for each core $k$,
the alternating linear systems approach~\cite{HOLTZ2012}  applied to \eqref{eq:FULLPROB} and alternating ridge regression~\cite{KLUS2019} are equivalent in this case.
As a consequence of the intrinsic nonlinearity of the overall iterative approach, the convergence analysis is complicated.
We refer to \cite{HOLTZ2012, Rohwedder2013,  Espig2015} for further information.

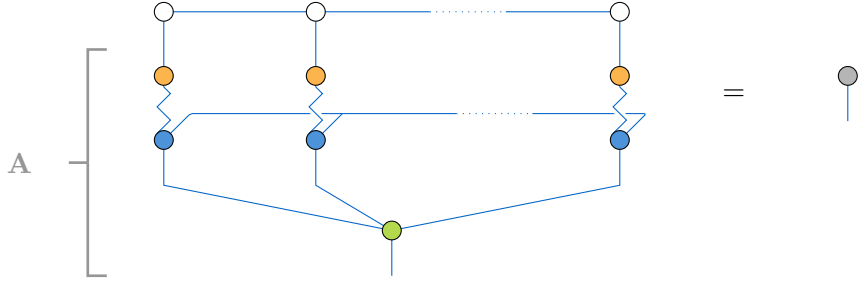
\begin{figure}[htb]
	\centering
	\begin{tikzpicture}
		\def\s{0.35}
		\draw[black, rounded corners = 1pt] (0,0) --++ (\s,\s) --++ (3.5,0);
		\draw[black] (2,0) --++ (\s,\s);
		\draw[black, dotted] (3.5+\s,\s) --++ (1,0);
		\draw[black, rounded corners = 1pt] (4.5+\s,\s) --++ (1.5,0) --++ (-\s,-\s);
		
		\draw[black] (0,1.7) --++ (3.5,0);
		\draw[black, dotted] (3.5,1.7) --++ (1,0);
		\draw[black] (4.5,1.7) --++ (1.5,0);
		
		\draw[black] (0,0.85) --++ (0,0.85);
		\draw[black] (2,0.85) --++ (0,0.85);
		\draw[black] (6,0.85) --++ (0,0.85);
		
		\draw[black,decoration={zigzag}, decorate] (0,0) --++ (0,0.85) node [label={[text depth=0]}] {};
		\draw[white,decoration={zigzag}, decorate, line width=3pt] (2,0) --++ (0,0.85);
		\draw[white,decoration={zigzag}, decorate, line width=3pt] (6,0) --++ (0,0.85);
		\draw[black,decoration={zigzag}, decorate] (2,0) --++ (0,0.85);
		\draw[black,decoration={zigzag}, decorate] (6,0) --++ (0,0.85);
		
		\draw[black] (0,0) --++ (0,-0.6) -- (3,-1.2);
		\draw[black] (2,0) --++ (0,-0.6) -- (3,-1.2);;
		\draw[black] (6,0) --++ (0,-0.6) -- (3,-1.2);;
		
		\draw[black] (3,-1.2) --++ (0,-0.6);
		
		\node[draw,shape=circle,fill=white, inner sep=2.5pt] at (0,1.7){};
		\node[draw,shape=circle,fill=white, inner sep=2.5pt] at (2,1.7){};
		\node[draw,shape=circle,fill=white, inner sep=2.5pt] at (6,1.7){};
		
		\node[draw,shape=circle,fill=Orange!70, inner sep=2.5pt] at (0,0.85){};
		\node[draw,shape=circle,fill=Orange!70, inner sep=2.5pt] at (2,0.85){};
		\node[draw,shape=circle,fill=Orange!70, inner sep=2.5pt] at (6,0.85){};
		
		\node[draw,shape=circle,fill=black!70, inner sep=2.5pt] at (0,0){};
		\node[draw,shape=circle,fill=black!70, inner sep=2.5pt] at (2,0){};
		\node[draw,shape=circle,fill=black!70, inner sep=2.5pt] at (6,0){};
		
		\node[draw,shape=circle,fill=Green!70, inner sep=2.5pt] at (3,-1.2){};
		
		\node[] at (7.5,0.6){$=$};
		
		\draw[black] (9,0.85) --++ (0,-0.6) ;
		\node[draw,shape=circle,fill=Gray!70, inner sep=2.5pt] at (9,0.85){};
		\def\y{-0.3}
		\node[] at (-1.9,\y){\textcolor{Gray}{$\mathbf{A}$}};
		\draw[Gray, line width=1pt] (-1.25,\y) --++ (0.25,0);
		\draw[Gray, line width=1pt] (-1,\y) --++ (0,1.5) --++ (0.25,0);
		\draw[Gray, line width=1pt] (-1,\y) --++ (0,-1.5) --++ (0.25,0);
		
	\end{tikzpicture}
	\caption{Graphical notation of tensor-based counterpart of the (underdetermined) system $A U = b$: The cores of $\mathbf{\Psi}$ (black circles) are contracted with the cores of $\mathbf{\Phi}$ (orange circles) by integrating over common modes. Together with the delta tensor (green circle), this system builds the tensor operator $\mathbf{A}$. The coefficient tensor $\mathbf{U}$ (white circles) is approximated in the TT format. Here, $\hat{\Psit}$ is given in canonical format but could also be defined as functional tensor train, see Section~\ref{sec: tensor factorization}.}
	\label{fig: system}
\end{figure}


\section{Sampling of discrete neural network parameters} \label{sec:SAMPLE}
By construction, an approximate solution $\tilde{\Ut}\in \R^{n_1\times \cdots \times n_s}$ of Problem~\ref{prob:TTAPROX} 
	provides an approximate solution 
	\[
	\tilde{U} = (\tilde{U}_j)\in \R^n, \qquad \tilde{U}_{j(j_1, \dots, j_s)} = \tilde{\Ut}_{j_1,\dots,j_s},
	\]
	of Problem~\ref{prob:ALGPROB}  by utilizing the enumeration \eqref{eq:ENUM}. In turn, $\tilde{U}$ gives rise to the approximate solution
	\[
	\tilde{u} = \sum_{j=1}^n \tilde{U}_j \varphi_j \in L_n \subset L^2(Q)
	\]
	of Problem~\ref{prob:DISCOPT} (function data) or Problem~\ref{prob:PPOINT} (point data) and the resulting approximation
	\begin{equation} \label{eq:FNT}
		\cG \tilde{u} = \sum_{j=1}^n  \tilde{U}_j \int_Q \psi(\cdot, \eta) \varphi_j(\eta) \; d \eta
	\end{equation}
	of the trained Fredholm network $\cG u$ introduced in Section~\ref{subsec:FN}. 
	Note that $\cG \tilde{u}(x)$ can  be  evaluated in given  test or training points by 
	contracting  $\tilde{\Ut}$ with the corresponding stiffness tensor $\mathbf{A}$.

Further approximation of $\cG \tilde{u}$ by classical Monte Carlo sampling of  the integral 
\begin{equation} \label{eq:DISCOPT}
	\cG \tilde{u} = \int_{Q}\psi(\cdot,\eta) \tilde{u}(\eta)\; d\eta
\end{equation}
provides  parameters $\tilde{\zeta}_i = (\tilde{\eta}_i, \tilde{u}_i)$ with  i.i.d.~$\tilde{\eta}_i$ and
$\tilde{u}_i  = \tilde{u}(\tilde{\eta}_i)$, $i=1,\dots,N$,  and thus the trained neural network $F_N(\cdot, \tilde{\zeta})$.

 {
\begin{remark} \label{rem:TRAINED}
Let us consider the situation addressed in Subsection~\ref{subsec:TIKHONOV}.
Under the (unrealistic) assumption that the linear Problem~\ref{prob:ALGPROB}  is solved exactly, 
classical Monte Carlo sampling of $\cG u_{\varepsilon, h}$ provides the  trained neural network $F_N(\cdot, \zeta_{\varepsilon, h})$.
Then, in the light of  Proposition~\ref{prop:RESERR},  the  residual error estimate
\[
\cJ_N(\zeta_{\varepsilon, h}) =\cO(\varepsilon^2 + \varepsilon^{-2}h^2 + N^{-1/2})
\]
for the discrete loss functional $\cJ_N$ introduced in \eqref{pval} holds by the central limit theorem under suitable regularity assumptions.
\end{remark}
}
Quasi-Monte Carlo sampling is based on quasi-random sequences in $\R^s$, such as, e.g., Sobol sequences,
that fill the space more uniformly than pseudo-random sequences as used for classical Monte Carlo integration.
In order to further enhance sampling quality, we now introduce  {\em inductive importance sampling}
that is directly targeting the underlying loss functional $\cJ_N$ defined in~\eqref{pval}.
Given a set of quasi-Monte Carlo sample points $(\tilde{\eta}_i)_{i=1}^{N^\prime}$,
corresponding parameter function evaluations $( \tilde{u}(\tilde{\eta}_i))_{i=1}^{N^\prime}$,  and some $N \leq  N^\prime$,
we want to inductively  find parameters $(\tilde{\eta}_{i_j})_{j=1}^N$ such that $F_N ( \cdot, \zeta)$ 
with $\zeta = (\tilde{\eta}_{i_j}, \tilde{u}(\tilde{\eta}_{i_j}))_{j=1}^N$ is reducing the loss functional $\cJ_N$ sufficiently well. 
For this purpose,  the first sample $\eta_{i_1}$ is chosen according to
\[
		\eta_{i_1}= \argmin_{\eta \in \{\eta_1, \dots, \eta_{N^\prime} \}}  \cJ_1((\eta, \tilde{u}(\eta)))
\]
and for given samples $\eta_{i_1}, \dots, \eta_{i_K}$, $K<N$, we select
\[
		\eta_{i_{K+1}} =  \argmin_{\eta \in \{\eta_1, \dots, \eta_{N^\prime}\}\setminus \{\eta_{i_1}, \dots, \eta_{i_K} \}} 
		\cJ_{K+1} \left( (\eta_{i_1}, \tilde{u}(\eta_{i_1})), \dots,  (\eta_{i_K}, \tilde{u}(\eta_{i_K})), (\eta, \tilde{u}(\eta) ) \right) .
\]
Note that we  use exact kernel and activation functions here and not their approximate factorizations which the Fredholm network is trained on.
While quasi-Monte Carlo integration performed reasonably well in simple cases, 
we found a considerable improvement of efficiency by inductive importance sampling in our numerical experiments to be reported below.

\section{Numerical experiments} \label{sec:NUMEX}

%
%
%
%

In order to demonstrate the practical relevance of our approach, we now consider three well-established test cases of regression and classification type.
In particular, we will numerically investigate different aspects of the tensor-based training of neural networks, such  as adaptability, accuracy, and robustness.
We emphasize that kernel, activation and even basis functions  for discretization have been chosen ad hoc, and no profound parameter optimization is performed. 
Our implementation of ARR and other relevant tensor algorithms are available in Scikit-TT\footnote{\url{https://github.com/PGelss/scikit_tt}}.

\subsection{Bank note authentication}
As a first example, we consider a classification problem to distinguish  forged from genuine bank notes, 
utilizing the UCI banknote authentication data set\footnote{\url{https://archive.ics.uci.edu/ml/datasets/banknote+authentication}} from~\cite{Lohweg2013}. 
It contains $1372$ samples of different features that were extracted from images of genuine and forged banknote-like specimens by applying wavelet transforms.
Each sample has $d=4$ different attributes: variance, skewness, curtosis, and entropy of the wavelet transformed images.
The data set is randomly divided into $M=1098$ ($\approx 80\%$) training and $m=274$ test samples.
We first apply min-max normalization to the given training data points $\hat{x}_i =(x_{i,k})_{k=1}^d$, $i=1, \dots, M$, to obtain
\begin{equation} \label{eq:MINMAXN}
	x_{i,k} = \frac{\hat{x}_{i,k} - \min(k)}{\max(k) - \min(k)} \in [0,1], \qquad  i=1, \dots, M, \quad k = 1,\dots, d,
\end{equation}
denoting  $\min(k) = \min_{i \in \{1, \dots, M\}} \hat{x}_{i,k}$  and $\max(k) = \max_{i \in \{1, \dots, M\}} \hat{x}_{i,k}$.
Hence, $X = [0,1]^d$. The test data points are normalized in the same way using the same constants.

For the neural network $F_N$ of the form \eqref{eq:PARREP},
we then choose a Gaussian activation function~\eqref{eq:GAUSSAC} with $\kappa=1$ and the same parameter set for each dimension, i.e., $Q = [-1,1]^s$ with $s=d=4$.

Discretization of the continuous training Problem~\ref{prob:FFUNC} is performed by piecewise constant finite elements 
with respect to an equidistant grid with mesh size $1/4$, see Example~\ref{ex:1}.
The resulting ansatz space has the dimension $n = 8^d = 4096$.
We select the  Tikhonov regularization parameter $\varepsilon = 10^{-3}$.
For the iterative solution of the resulting Problem~\ref{prob:TTAPROX} with  ranks $r_1= \cdots = r_{s-1}=10$ 
we apply $20$ sweeps of ARR (cf.~Section~\ref{sec: ARR})  with  all tensor core entries  equal to one as initial guess.  
From the resulting approximation $\tilde{\Ut}$ of the exact solution $\Ut_r$,
we then obtain the approximation $\tilde{u}\in L _n \subset L^2(Q)$ of the solution $u\in L^2(Q)$ of Problem~\ref{prob:FFUNC},
and the resulting approximately trained Fredholm network $\cG \tilde{u}$ as explained in Section~\ref{sec:SAMPLE}.
 From $\cG \tilde{u}$ we  extract  discrete parameters $\tilde{\zeta} = (\tilde{\eta}_i, \tilde{u}(\tilde{\eta}_i) )_{i=1}^N$
by quasi-Monte Carlo sampling for various numbers $N$ of neurons to obtain the trained neural network $F_N(x, \tilde{\zeta})$, cf.~Section~\ref{sec:SAMPLE}.
A  sample $x\in X$ is then classified as $0$ (genuine), if $F_N(x, \tilde{\zeta}) < 0.5$ and $1$ (forged) otherwise. 

In order to test the predictability of our trained neural networks for a fixed number $N$ of neurons,
we randomly select training and test sets of as described above, use quasi-Monte Carlo sampling to extract discrete parameters and evaluate the corresponding realization of the classification rate, i.e., the ratio of the number of correct predictions and the size of the selected test set.
This procedure is repeated 100 times, in order to obtain expectation and standard deviation of the classification rates
together with the number of perfect classifications, i.e., the number of test sets that are correctly classified without an exception.
The results are shown in Table~\ref{tab: banknote} for various $N$.
As expected, the predictability is increasing with increasing $N$.
While no test sets are perfectly classified for $N \leq 256$, we even achieve  full optimality for $N=1024$.
Note that predictability deteriorates for even larger $N$, which might be due to numerical noise caused by round-off errors.
\begin{table}[]
	\caption{Results for the bank note authentication data set: Predictability of trained neural network $F_N(\cdot, \tilde{\zeta})$ 
	for increasing~$N$ in terms of expectation and standard deviation of the classification rates.}
	\begin{tabular}{ccc}
		$N$ & classification rate & perfect classification     \\\hline\\[-0.3cm]
		$2^6 = 64$      & $0.9353 \pm 0.0147$ & 0\\
		$2^7 = 128$     & $0.8653 \pm 0.0279$ & 0\\
		$2^8 = 256$     & $0.7827 \pm 0.0136$ & 0 \\
		$2^9 = 512$     & $0.9991 \pm 0.0022$ & 81 \\
		$2^{10} = 1024$ & $1.0000 \pm 0.0000$ & 100 \\
		$2^{11} = 2048$ & $0.9999 \pm 0.0007$ & 96
	\end{tabular}
	\label{tab: banknote}
\end{table}

\subsection{Concrete compressive strength prediction}
We aim at the description of concrete compressive strength  in terms of tolerated mechanical stress
as a function $F$ of the values $x=(x_k)_{k=1}^d$ of  $d=8$ parameters, i.e.,  the percentage of 
cement, fly ash, blast furnace slag, superplasticizer, fine aggregates, coarse aggregates, water, and age.
To this end, we make use of a data set with 1030 samples  that was collected over several years 
and has repeatedly been deployed for corresponding training of neural networks, see, e.g.,~\cite{YEH1998, Yeh2003, Yeh2006, Shah2022}.
We randomly split the data set into $M = 824$ ($80 \%$) training and $m = 206$ test samples.
As the output of each sample is distributed between $2.33\,\mathrm{MPa}$ and $82.60\,\mathrm{MPa}$, different  input variables might be of different orders of magnitude.
In order to  avoid large values dominating the others,  
min-max normalization is applied in analogy to \eqref{eq:MINMAXN}, and we obtain $X=[0,1]^d$.
 
We consider a neural network of the form \eqref{eq:PARREP} with ridge kernel and Heaviside activation function~\eqref{eq:HEAVI}
together with the parameter set $Q = [-1,1]^s \subset \R^s$, $s=d+1=9$.

For  discretization of the continuous training Problem~\ref{prob:FFUNC}, we choose the ansatz space $L_n$ spanned by  $n = 3^d = 6561$  
tensor products of Chebyshev polynomials of the first kind up to order $2$. 
The regularization parameter is set  to $\varepsilon=10^{-8}$.
Approximate factorization of the Heaviside activation function is performed as
presented in the Appendix~\ref{A:HEAVI}, and we set $\ell_0 = 100$. 
Hence, the decomposition \eqref{eq:AFAC} of the corresponding approximate coefficient tensor 
$\At \in  \R^{M \times n}= \R^{824 \times 6561}$  has the canonical rank $r=201$. 
We normalize the tensor $\At$ such that $\lVert \At \rVert_{2,\R^{M \times n}} = 1$. 
The coefficient tensor corresponding to the test set is multiplied by the same normalization constant.

For the tensor-train approximation \eqref{eq:TTAPPROX} in Problem~\ref{prob:TTAPROX} we choose the ranks $r_1= \cdots = r_{s-1}=10$
and  apply 5 sweeps of ARR (cf.~Section~\ref{sec: ARR}) with  all tensor core entries  equal to one as an initial guess.
From  the resulting  approximate solution  $\tilde{\Ut}$ we derive the approximate solution $\tilde{u}$ 
of Problem~\ref{prob:FFUNC} which in turn provides the approximately trained Fredholm network $\cG \tilde{u}$ as explained in Section~\ref{sec:SAMPLE}. 
Discrete parameters $\tilde{\zeta} = (\tilde{\eta}_i, \tilde{u}(\tilde{\eta}_i) )_{i=1}^N$ 
providing  the corresponding neural network $F_N(\cdot, \tilde{\zeta})$ can finally be obtained by a suitable sampling strategy.
In this experiment, we use  both standard quasi-Monte Carlo sampling and inductive importance sampling  as introduced in Section~\ref{sec:SAMPLE}.

The prediction accuracy of the Fredholm network $\cG \tilde{u}$ and of neural networks $F_N(\cdot, \tilde{\zeta})$
on given test data $(y_i, F(y_i))$, $i=1,\dots,m$, are  measured by the usual $R^2$ scores
\[
	 R^2_{\cG} = 1 - \frac{\sum_{i=1}^{m} (F(y_i) - (\cG \tilde{u})(y_i))^2}{\sum_{i=1}^{m} (F(y_i)  - \overline{F}(y))^2} \quad \text{and} \quad 	
	 R^2_N = 1 - \frac{\sum_{i=1}^{m} (F(y_i) - F_N(y_i,\tilde{\zeta}))^2}{\sum_{i=1}^{m} (F(y_i)  - \overline{F}(y))^2},
\]
respectively, where $\overline{F}(y) = \frac{1}{m}\sum_{i=1}^m F(y_i)$ stands for the algebraic mean of the test values.
Note that for  evaluation of $F_N(y_i,\tilde{\zeta})$ the exact Heaviside activation function $\sigma_{\infty}$ is used  and not its approximate factorization.
Corresponding $R^2_{\cG}$ and $R^2_N$ scores on training data are computed in the same way using the training data $(x_i, F(x_i))$, $i = 1,\dots, M$, instead.
Similar to the previous example, $R^2$ scores are considered as random variables, and we apply standard Monte Carlo with 
100 samples each of random training and test  data sets to approximately compute its expectation and standard deviation.

\begin{table}[]
	\caption{Results for the Concrete compressive strength prediction: Predictability of trained neural network $F_N(\cdot, \tilde{\zeta})$ for increasing~$N$ 
	in terms of expectation and standard deviation of  $R^2$ scores.} 
	\begin{tabular}{ccc}
		$N$ & $R^2_N$ score (training data) & $R^2_N$ score (test data)\\\hline\\[-0.3cm]
		$200$     & $0.9023 \pm 0.0051$ & $ 0.8040 \pm 0.0244$\\
		$400$     & $0.9406 \pm 0.0027$ & $ 0.8330 \pm 0.0219$\\
		$600$     & $0.9563 \pm 0.0023$ & $ 0.8473 \pm 0.0209$\\
		$800$     & $0.9648 \pm 0.0019$ & $ 0.8532 \pm 0.0204$\\
		$1000$    & $0.9702 \pm 0.0019$ & $ 0.8586 \pm 0.0198$ 
	\end{tabular}
	\label{tab: ccs}
\end{table}
The resulting scores of the Fredholm network $\cG\tilde{u}$ are
\[
 \text{training sets:}\;\; R^2_\cG = 0.9154 \pm 0.0029 \qquad  \qquad  \text{ test sets:}\;\; R^2_\cG =  0.8701 \pm 0.0180.
\]
with  samples of the $R^2_\cG$ scores on training sets ranging between $0.9093$ and $0.9244$ and on test sets  between $0.8202$ and $0.9058$.
These results indeed match the scores obtained by state-of-the-art ML approaches, cf.~\cite{Shah2022, Young2019}.


However, in order to achieve similar $R^2_N$ scores for  discrete neural networks $F_N(\cdot, \tilde{\zeta})$ 
as obtained  
by  quasi-Monte Carlo sampling,
it turned out  in our numerical experiments that about $N=2^{20} \approx 10^6$  sampling points were required.
This motivates application of more advanced  sampling strategies such as  inductive importance sampling  introduced in Section~\ref{sec:SAMPLE}.
For this  strategy, we obtained a considerable reduction of  the number $N$ of required samples,  as illustrated in Table~\ref{tab: ccs}.
As in the previous example,  predictability increases with growing $N$,  both on training and test data. 
Observe that the  $R^2_N$ scores on training data are even better than for the Fredholm network $\cG\tilde{u}$
which might be an outcome of the particular importance of training data for the applied sampling strategy.
Even though the  $R^2_N$ scores on  test sets are slightly worse,
they are still comparable with  state-of-the-art ML approaches, cf.~\cite{Shah2022, Young2019}.

\subsection{MNIST data set}
In our third experiment, we classify hand-written digits from the MNIST data set~\cite{Lecun1998} 
that consists of representations of the ten digits  $0, \dots, 9$ as  grayscale images of $28 \times 28$ pixels.
We slightly modified the MNIST dataset by reducing the size of these images to $d=14 \times 14=196$ pixels 
in order to lower the computational effort for classification.
The data set is divided into  $M=60,000$ training images and $m=10,000$ test images with associated labels,  and we have $X=[0,1]^d$.

For the neural network, we choose the ridge kernel $\psi$ from \eqref{eq:RIGDE} 
with logistic activation function $\sigma_\kappa$ from \eqref{eq:LOGISTIC} and $\kappa=1$.
More precisely, we use its FTT approximation based on the identification of the sample space $X$ with its embedding $\{(x,1)^\top\;|\; x \in X\}\subset \R^s$, $s=d+1=197$ and Taylor approximation~\eqref{eq:LOGTAYLOR} at $z=0$. See the Appendix~\ref{A:LOGISTIC} for details.
In order to provide sufficient accuracy of this approximation, we require $|z|=| x_i \cdot \eta | \leq 2$ for all data points $x_i \in X$.
This is guaranteed by the choice  $\eta \in Q = [-2/\rho_X, 2/\rho_X]^s$, denoting $ \rho_X = \max_{i \in \{1, \dots, M\}} \sum_{j=1}^{d+1} x_{i,j}$.

As each image $x\in X$ has to be classified as one of the ten digits, the target function $F$  is now vector-valued.
Utilizing one-hot encoding, we get  $F(x) = (F_i(x))_{i=1}^{10}\in \R^{10}$, $x \in X$, with
\begin{equation} \label{eq:LABELLING}
F_l(x) = \begin{cases}
		1 & \text{if } x \text{ represents the digit } l \\
		0 & \text{otherwise}
	\end{cases} , \quad l =0, \dots,9.
\end{equation}
Corresponding vector-valued versions of the neural network \eqref{eq:PARREP} and of its continuous counterpart \eqref{prob:PPOINT} 
are obtained by taking coefficients $u_i \in \R^{10}$ and parameter functions $u \in \left(L^2(Q)\right)^{10}$, respectively.
Observe that each component $u_l$ of $u$ can be (approximately) computed separately 
from the scalar Problem~\ref{prob:PPOINT} with $F=F_l$ and $F_l$ defined in \eqref{eq:LABELLING}.

For the Ritz-Galerkin discretization of each of these scalar problems, we choose basis elements $\varphi^{(k)}_j$ of the form \eqref{eq:PROD} 
with constant  and linear functions $\varphi^{(k)}_{j_1}$ and  $\varphi^{(k)}_{j_2}$, respectively.
This leads to a subspace $L_n$ of $L^2(Q)$ with dimension $n = 2^s = 2^{197} \approx 2 \cdot 10^{59}$.
We choose the Tikhonov regularization parameter $\varepsilon = 10^{-15}$
and fix $r_1 = \dots = r_d = 20$ in the tensor train approximation Problem~\ref{prob:TTAPROX}.
To the initial iterate of all core entries set to one (before orthonormalization) we apply $5$ ARR sweeps 
to obtain the approximate solution $\tilde{\Ut}_l$, which  provides  the corresponding component  $\cG \tilde{u}_l$ of the (approximately)
 trained Fredholm network, cf.~Section~\ref{sec:SAMPLE}.
We finally use the softmax function for classification, i.e. for each entry $x$ the index $l_{\max}$ 
of the largest of the components  $(\cG \tilde{u}_l) (x)$, $l=0,\dots,9$,  determines the detected label.

For comparison with this trained Fredholm network,
we  consider single layer neural networks built from the well-established Keras library\footnote{\url{https://keras.io}}.
The input layer of the network comprises $196$ nodes (one for each pixel), followed by a hidden layer 
with  a varying number of $N_K$ nodes and the logistic activation function. 
The output layer also uses one-hot encoding and the softmax function.
For optimizing of the neural network parameters, $10\%$ of the given training set is used as a validation set.
The parameters are then trained for a sufficiently large number of iterations in order to ensure convergence 
of the validation loss which indicates how well the model fits to unseen data.

\begin{figure}[htb]%
	\centering%
	\includegraphics[width=8cm]{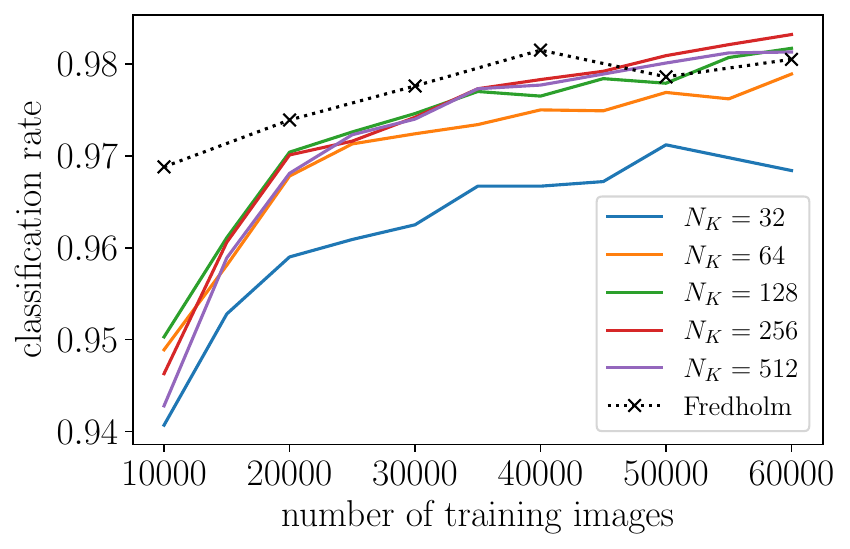}
	\caption{Results for the MNIST data set:
	Classification rates obtained from the  Fredholm network over the number $M$ of  training images 
	 in comparison with  Keras networks with increasing number $N_K$ of nodes in the hidden layer.}
	\label{fig: MNIST}
\end{figure}

Figure~\ref{fig: MNIST} shows the classification rates as obtained from our Fredholm network $\cG \tilde{u}$ and of  Keras networks 
of size $N_K= 34, 64, 128, 256, 512$ on the fixed MNIST test data set of $m=10,000$ test images over the number $M$ 
of  utilized training images.  
For the largest number of $M=60,000$ training data,
the classification rates of the Keras networks are ranging from  $0.9789$ to $0.9832$ for sufficiently large $N_K \geq 64$
and are thus comparable with the classification rate $0.9805$ of the Fredholm network.
However,  Keras networks are clearly outperformed for smaller sets of training data.

\section{Conclusion and outlook} \label{sec:CONCOUT}
In this work, we have proposed a novel approach to function approximation and the training of large and high-dimensional shallow neural networks,
based on their continuous asymptotic limit.
Utilizing  resulting Fredholm networks  for function approximation and also for training in the original discrete case,
we have thus traded highly nonlinear discrete training problems with finitely many unknown parameters 
for linear continuous training problems with infinitely many unknown values of a parameter function. 

Fredholm training problems can be regarded as least-squares formulations of Fredholm integral equations of the first kind.
 {Tikhonov regularization, approximation by Ritz-Galerkin methods and exact solution of the resulting linear system
provide residual error estimates for the underlying loss functional.
However, exact linear solution is unrealistic for most  practical problems.
Therefore,} we particularly aimed at a significant reduction of  memory consumption as well as  computational costs 
to mitigate the curse of dimensionality as occurring for high-dimensional data. 
To this end, we described different tensor formats, suitable factorizations of kernel and basis functions
as well as an alternating scheme for solving the tensorized linear training problems.
Finally, we considered quasi Monte-Carlo sampling of discrete neural network parameters,
and  introduced inductive importance sampling  which is directly targeting the loss functional $\cJ_N$.
 {Note that residual error estimates extend to the trained neural network in the (unrealistic) case of exact linear solution.}

The predictive quality of the resulting approximatively trained Fredholm and neural networks 
are illustrated by numerical experiments with three well-established test problems of regression and classification type.
The results clearly confirm that,  {in spite of the apparently strong constraints \eqref{eq:ADHOC},  the presented Fredholm} approach is highly competitive with state-of-the-art neural network-based methods concerning efficiency and reliability.

 {
These promising results suggest future investigation of Fredholm-type counterparts of  deep neural networks.
For example, let $\sigma$ be some continuous activation function and $\zeta^{(1)} = \left( U^{(1)}, \;W^{(1)},\;  b^{(1)} \right)$ with
\[
W^{(1)} = \left(w_i^{(1)}\right)_{i=1}^{N^{(1)}} \in \R^{N^{(1)} \times d}, \;\;\; b^{(1)} =  \left(b_i^{(1)}\right)_{i=1}^{N^{(1)}} \in \R^{N^{(1)}},  \;\;\;
U^{(1)} =  \frac{1}{N^{(1)}}\left(u_1^{(1)},\dots, u_{N^{(1)}}^{(1)}\right)   
\]
and $b^{(2)} \in \R$ be the parameters of the deep neural network 
\[
F_{N^{(1)}}\left(x, \zeta^{(1)}\right) = \sigma\left( U^{(1)} \sigma \left(W^{(1)}x + b^{(1)} \right) + b^{(2)}\right) = 
\sigma\left( \frac{1}{N^{(1)}} \sum_{i=1}^{N^{(1)}} u_i^{(1)} \sigma\left(  w_i^{(1)}\cdot x + b^{(1)}_i\right) + b^{(2)} \right)
\]
with input $x \in \R^d$, output $F_{N^{(1)}}\left(x, \zeta^{(1)}\right)  \in \R$, one hidden layer with $N^{(1)}$ neurons,
and another hidden layer with only one neuron for simplicity. Then we obtain
\[
F_{N^{(1)}}\left(x, \zeta^{(1)}\right) 
\quad \to \quad \sigma \left(  \cG^{(1)} u^{(1)}(x) + b^{(2)}\right) \quad \text{for} \quad N^{(1)} \to \infty
\]
with 
\[
\cG^{(1)}v (x)= \frac{1}{|Q^{(1)}|}\int_{Q^{(1)}}  v(\eta) \sigma(w\cdot x + b)\; d \eta \in \R, \qquad \eta =(w,b) \in Q^{(1)}\subset \R^{d+1},
\]
by the reasoning presented in Section~\ref{subsec:ASIOP}. 
}
 {
A  counterpart of \eqref{eq:PARREP} with two hidden layers is given by
\[
F_{N^{(1)},N^{(2)}}\left(x,\zeta^{(1)}, \zeta^{(2)}\right) = \frac{1}{N^{(2)}} \sum_{i=1}^{N^{(2)}} u_i^{(2)} \sigma \left( w_i^{(2)} U^{(1)} \sigma \left(W^{(1)}x + b^{(1)}\right) + b_i^{(2)} \right)
\]
with additional parameters $\zeta^{(2)} = (\zeta^{(2)}_i )_{i=1}^{N^{(2)}}$ and $\zeta_i^{(2)} = (u_i^{(2)}, w^{(2)}_i , b^{(2)}_i ) \in \R^{3}$. 
Note that $W^{(2)}= (w_i^{(2)}U^{(1)})_{i=1}^{N^{(2)}} \in \R^{N^{(2)} \times N^{(1)}}$ is a rank-one matrix.
The reasoning presented in Section~\ref{subsec:ASIOP} then yields
\[
F_{N^{(1)},N^{(2)}}\left(x,\zeta^{(1)}, \zeta^{(2)}\right) \quad \to \quad \cG^{(2)}u^{(2)}\left( \cG^{(1)} u^{(1)}(x) \right) \quad \text{for} \quad N^{(1)},\; N^{(2)}\to \infty
\]
with the additional Fredholm operator 
\[
\cG^{(2)}v(y) = \frac{1}{|Q^{(2)}|}\int_{Q^{(2)}}  v(\eta) \sigma(w y + b)\; d \eta \in \R, \qquad \eta =(w,b) \in Q^{(2)}\subset \R^{2}.
\]
Deep Fredholm networks with more than two hidden layers are obtained in a similar way.
While linearity of the minimization of  corresponding loss functionals is clearly lost for deep Fredholm networks, 
the construction of appropriate Tensor-Gau\ss -Newton methods
might still benefit from structural properties such as the monotonicity of activation functions or the block structure of iterated Fredholm operators.
}

 {
Both shallow and deep Fredholm networks are open to promising investigations concerning 
analysis, numerical analysis, and algorithmic improvements, which might also shed some light 
on structural properties of the approximating neural networks.
}


\appendix
\section{Factorization of ridge kernel functions}\label{app}
%
%
According to \eqref{eq:RIGDE}, ridge kernel functions $\psi \colon X \times Q \to \R$, are given by
\[
 \psi(x, \eta ) = \sigma(w \cdot   x + b), \quad \eta = (w,b) \in Q \subset \R^s, \; s=d+1.
 \]
with hypercubes $X \in \R^d $, $Q= Q_1 \times \cdots \times Q_s\subset \R^s$, and ridge activation function $\sigma: \R \mapsto \R$.
Identifying $X$ with its  embedding $\{(x,1)^\top\;|\; x \in X\}\subset \R^s$, ridge kernel functions take the form 
\[
 \psi(x, \eta ) = \sigma(x \cdot \eta).
 \]
In what follows, explicit tensor decompositons of different ridge activation functions are derived, either using the canonical format or TT format.
In case of the latter, we use the core notation of the TT format~\cite{gelss2017nearest} to represent FTT cores as two-dimensional arrays containing vector-valued functions as elements.
More precisely, for a given functional tensor train $\hat{\Psit}$ with cores $\hat{\Psit}^{(k)} \in \Cc^{r_{k-1} \times M \times Q_k  \times r_k}$ defined by 
\[
	\hat{\Psit}^{(k)}_{\ell_{k-1}, i, \eta_k, \ell_k} = \psi^{(k)}_{\ell_{k-1}, \ell_k} (x_{i,k}, \eta_k) \colon \R \times Q_k \to \Cc,
\]
a single core is written as
\begin{equation} \label{eq:CORE}
	\core{\hat{\Psit}^{(k)}} = \core{ & &  \\[-0.3cm] \hat{\Psit}^{(k)}_{1, :, \eta_k, 1}  & \cdots & \hat{\Psit}^{(k)}_{1, :, \eta_k, r_k} \\ \vdots & \ddots &\vdots \\ \hat{\Psit}^{(k)}_{r_{k-1}, :, \eta_k, 1} & \cdots & \hat{\Psit}^{(k)}_{r_{k-1}, :, \eta_k, r_k} \\[-0.35cm] & &} ,
\end{equation}
where $\hat{\Psit}^{(k)}_{\ell_{k-1}, :, \eta_k, \ell_k} = ( \psi^{(k)}_{\ell_{k-1}, \ell_k}(x_{i,k}, \eta_k) )_{i=1}^M$.
Thus, the FTT decomposition of the auxiliary tensor $\hat{\Psit}$, cf.~\eqref{eq:HELP}, takes the form  $\hat{\Psit} = \core{\hat{\Psit}^{(1)}} \otimes \dots \otimes \core{\hat{\Psit}^{(s)}}$, denoting
\[
\left(\core{\hat{\Psit}^{(k)}}\otimes \core{\hat{\Psit}^{(k+1)}}\right)_{\ell_{k-1},i_k,\eta_k, i_{k+1}, \eta_{k+1} ,\ell_{k+1}} = \sum_{\ell_k = 1}^{r_k} \hat{\Psit}^{(k)}_{\ell_{k-1},i_k,\eta_k, \ell_k} \cdot \hat{\Psit}^{(k+1)}_{\ell_{k}, i_{k+1}, \eta_{k+1} ,\ell_{k+1}}.
\]
\begin{proposition}\label{prop: monomial psi}
	Consider a monomial of the form $\psi(x, \eta) = (x \cdot \eta)^p$, $p \geq 0$, with $\eta \in Q$ and data points $x_i \in X$, $i=1, \dots, m$.
	The functional tensor $\Psit \in \R^{M \times Q_1 \times \cdots \times Q_s}$ defined by $\Psit_{i, \eta_1, \dots, \eta_s} = \psi(x_i, \eta)$ can be written as $\Psit = \Delta \cdot \hat{\Psit}$, $\hat{\Psit} \in \R^{M^s \times Q_1 \times \dots \times Q_s}$, where
	\begin{equation}\label{eq: monomial psi}
		\begin{split}
		\hat{\Psit} =&~p! \core{\frac{1}{p!} ( x_{:,1} \eta_1)^{p} & \cdots & x_{:,1} \eta_1 & e} \otimes \core{e & 0 & \cdots & 0 \\ x_{:,2} \eta_2 & \ddots & \ddots & \vdots\\ \vdots & \ddots & \ddots & 0 \\ \frac{1}{p!} ( x_{:,2}'\eta_2)^{p} & \cdots &  x_{:,2} \eta_2 & e} \otimes \cdots \\
		&~\otimes \core{e & 0 & \cdots & 0 \\  x_{:,s-1} \eta_{s-1} & \ddots & \ddots & \vdots\\ \vdots & \ddots & \ddots & 0 \\ \frac{1}{p!} ( x_{:,s-1} \eta_{s-1})^{p} & \cdots &  x_{:,s-1} \eta_{s-1} & e} \otimes \core{e  \\  x_{:,s} \eta_{s} \\ \vdots \\ \frac{1}{p!} ( x_{:,s} \eta_{s})^{p}}
		\end{split}
	\end{equation}
	with $e = (1, \dots, 1)^T \in \R^M$.
\end{proposition}
\begin{proof}
	By using the multinomial formula, we know
	\[
		\left( \sum_{k=1}^s z_k \right)^p = \sum_{\substack{\ell_1 + \dots + \ell_s = p \\ \ell_1 , \dots, \ell_s \geq 0}} \frac{p!}{\ell_1! \cdots \ell_s!} z_1^{\ell_1} \cdots z_s^{\ell_s}
	\]
	for any $i \in \{1, \dots, M\}$. This can be written as a sequence of matrix multiplications of the form 
	\[
		\begin{split}
	 \left( \sum_{k=1}^s z_k \right)^p =&~p! \begin{bmatrix} z_1^{p} & \cdots & z_1^{0} \end{bmatrix} 
	 	\begin{bmatrix} \displaystyle \sum_{\ell_2 + \dots + \ell_s = 0 } \frac{1}{\ell_2! \cdots \ell_s!} z_2^{\ell_2} \cdots z_s^{\ell_s}\\ \vdots \\
	 	\displaystyle \sum_{\ell_2 + \dots + \ell_s = p } \frac{1}{\ell_2! \cdots \ell_s!} z_2^{\ell_2} \cdots z_s^{\ell_s}  \end{bmatrix}\\
	 	=&~p! \begin{bmatrix} z_2^{p} & \cdots & z_2^{0} \end{bmatrix} 
	 	\begin{bmatrix}
	 		z_2^{0} &  & 0 \\
	 		\vdots & \ddots &  \\
	 		z_2^{p} & \cdots & z_2^{0}
	 	\end{bmatrix}
	 	\begin{bmatrix} \displaystyle \sum_{\ell_3 + \dots + \ell_s = 0 } \frac{1}{\ell_3! \cdots \ell_s!} z_3^{\ell_3} \cdots z_s^{\ell_s} \\ \vdots \\
	 	\displaystyle \sum_{\ell_3 + \dots + \ell_s = p } \frac{1}{\ell_3! \cdots \ell_s!} z_3^{\ell_3} \cdots z_s^{\ell_s} \end{bmatrix} \\
 	=&~\dots \\
 	=&~p! \begin{bmatrix} z_2^{p} & \cdots & z_2^{0} \end{bmatrix} 
 	\begin{bmatrix}
 		z_2^{0} &  & 0 \\
 		\vdots & \ddots &  \\
 		z_2^{p} & \cdots & z_2^{0}
 	\end{bmatrix} \cdots
 	\begin{bmatrix}
 		z_{s-1}^{0} &  & 0 \\
 		\vdots & \ddots &  \\
 		z_{s-1}^{p} & \cdots & z_{s-1}^{0}
 	\end{bmatrix}
 	\begin{bmatrix} z_s^0 \\ \vdots \\ z_s^p \end{bmatrix}.
	 	\end{split}
	\]
	Thus, for $\Psit = \Delta \hat{\Psit}$ with $\hat{\Psit}$ as given in~\eqref{eq: monomial psi}, we get
	\[
		\begin{split}
		\Psit_{i, \eta_1 , \dots , \eta_s} =&~\sum_{j_1=1}^M \dots \sum_{j_s=1}^M \Delta_{i,j_1, \dots, j_s}  \hat{\Psit}_{j_1 , \dots, j_s, \eta_1, \dots, \eta_s}\\
		=&~\sum_{j_1=1}^M \dots \sum_{j_s=1}^M \delta_{i,j_1} \cdots \delta_{i,j_s}  \hat{\Psit}_{j_1 , \dots, j_s, \eta_1, \dots, \eta_s}\\
		=&~\hat{\Psit}_{i, \dots, i, \eta_1, \dots, \eta_s}  = \left( \sum_{k=1}^s x_{i,k} \eta_k \right)^p  = (x_i \cdot \eta)^p.\qedhere
		\end{split}
	\]
\end{proof}

\subsection{Trivial activation function} \label{A:TRIV}

In the simplest case $\sigma = \text{Id}$, the ridge kernel function is bilinear, i.e., $\psi(x, \eta) = x \cdot \eta$, and clearly can be written in the form \eqref{eq:FACKER} (see Example~\ref{ex:FACKER}).
Let us consider the  functional tensor
\[
{\mathbf \Psi}\in \R^{M \times Q_1 \times \cdots \times Q_s}, \qquad  \mathbf{\Psi}_{i, \eta_1, \dots, \eta_s} = x_i \cdot \eta = \sum_{k=1}^s  x_{i,k} \eta_k, \qquad s=d,
\]
associated to $\psi$ according to \eqref{eq:FUNTENS} for the given data points $x_i \in X$, $i=1,\dots,M$.
It can be expressed as $\Psit = \Delta \cdot \hat{\Psit}$, see Proposition~\ref{prop:KERNEL}, where the auxiliary tensor $\hat{\Psit}$ has the functional canonical decomposition
\begin{equation}\label{eq: triv_canonical}
 \hat{\Psit} = \eta_1 \begin{pmatrix} x_{1,1} \\ \vdots \\ x_{M,1} \end{pmatrix} \otimes \begin{pmatrix} 1 \\ \vdots \\ 1 \end{pmatrix} \otimes \dots \otimes \begin{pmatrix} 1 \\ \vdots \\ 1 \end{pmatrix}  ~ + ~ \cdots ~ + ~ \begin{pmatrix} 1 \\ \vdots \\ 1 \end{pmatrix} \otimes \dots \otimes \begin{pmatrix} 1 \\ \vdots \\ 1 \end{pmatrix} \otimes \eta_d \begin{pmatrix} x_{1,d} \\ \vdots \\ x_{M,d} \end{pmatrix} 
\end{equation}
Using the core notation~\eqref{eq:CORE} and Proposition~\ref{prop: monomial psi}, we can also represent $\hat{\Psit}$ as a functional tensor train:
\begin{equation}\label{eq: triv_ftt}
 \hat{\Psit} = \core{x_{:,1} \eta_1 & e } \otimes \core{ e & 0 \\  x_{:,2} \eta_2 & e } \otimes \dots \otimes \core{ e & 0 \\  x_{:,d-1} \eta_{d-1} & e } \otimes \core{ e  \\  x_{:,d} \eta_{d} }.
\end{equation}
Note that the canonical rank of~\eqref{eq: triv_canonical} is equal to $d$ whereas the FTT ranks $r_1, \dots, r_{d-1}$ of~\eqref{eq: triv_ftt} are always equal to $2$ for any state space dimension $d$.

\subsection{Logistic activation function}\label{A:LOGISTIC}
For the logistic activation function $\sigma_\kappa$ given in \eqref{eq:LOGISTIC}, 
the  kernel function takes the form
\begin{equation} \label{eq:LOGTAYLOR}
 \psi_\kappa(x, \eta) = \frac{\exp(\kappa \,  x \cdot \eta)}{1 + \exp(\kappa \,  x \cdot \eta))} = \frac{1}{1+ \exp(-\kappa \,  x \cdot \eta)} = \frac{1}{2} + \frac{1}{2} \tanh \left(\frac{-\kappa \,  x \cdot \eta}{2} \right).
\end{equation}
We expand  $\sigma_\kappa(z)$ into a Taylor series at $z=0$,  to obtain
\begin{equation}\label{eq: logistic AF - Taylor}
 \psi_\kappa(x, \eta ) = \frac{1}{2} + \frac{\kappa \,  x \cdot \eta} {4} - \frac{\kappa^3 \,  (x \cdot \eta)^3} {48} + O( x \cdot \eta)^5.
\end{equation}
Assuming that the side lengths of $Q$ are sufficiently small, 
we skip the quintic part in \eqref{eq: logistic AF - Taylor}. 
The resulting approximation of $\psi_\kappa$ can be written in the form of \eqref{eq:FACKER} with $r= 1 + s + \frac{1}{6}(s^3 + 3 s^2 + 2 s)$.
Indeed, the zero- and the first-order term in  \eqref{eq: logistic AF - Taylor} can be written as a functional canonical tensor with rank $1 + s $ summands, cf.~\eqref{eq: triv_canonical},
while, by the multinomial formula, the cubic term can be written as
\[
\frac{\kappa^3 \,  (x \cdot \eta)^3} {48} =
 \frac{\kappa^3}{48} \left(\sum_{k=1}^s x_k\eta_k\right)^3 =  \frac{\kappa^3}{48} \sum_{\ell_1 + \dots + \ell_s =3} \frac{6}{\ell_1!\cdots \ell_s!}
(x_{1}\eta_{1})^{\ell_1}\cdots (x_{s}\eta_{s})^{\ell_s}
\]
with  
 $\sum_{k_3=1}^s \sum_{k_2=1}^{k_3} \sum_{k_1=1}^{k_2} 1 = \frac{1}{6}(s^3 + 3s^2 + 2s)$ summands.
In contrast to that, the ranks of the corresponding FTT decomposition are always bounded by $4$ for any parameter dimension $s$.
This can be seen by using Proposition~\ref{prop: monomial psi} for each of the first three terms in~\eqref{eq: logistic AF - Taylor} and adding the respective decompositions.
The resulting FTT representation of the auxiliary tensor is then given by 
\[
	\renewcommand*{\arraystretch}{1.2}
	\setlength{\arraycolsep}{2pt}
	\begin{split}
		\hat{\Psit} =&  \core{ \frac{1}{6}(\eta_1 x_{:,1})^3 & \frac{1}{2}(\eta_1 x_{:,1})^2 & \eta_1 x_{:,1} & e } \otimes \core{ e & 0 & 0 & 0 \\ \eta_2  x_{:,2} & e & 0 & 0 \\ \frac{1}{2}(\eta_2  x_{:,2})^2 & \eta_2  x_{:,2} & e & 0 \\ \frac{1}{6}(\eta_2  x_{:,2})^3 & \frac{1}{2}(\eta_2  x_{:,2})^2 & \eta_2  x_{:,2} & e } \otimes \dots \\
		&\otimes \core{ e & 0 & 0 & 0 \\ \eta_{d-1}  x_{:,{d-1}} & e & 0 & 0 \\ \frac{1}{2}(\eta_{d-1}  x_{:,{d-1}})^2 & \eta_{d-1}  x_{:,{d-1}} & e & 0 \\ \frac{1}{6}(\eta_{d-1}  x_{:,{d-1}})^3 & \frac{1}{2}(\eta_{d-1}  x_{:,{d-1}})^2 & \eta_{d-1}  x_{:,{d-1}} & e } \otimes \core{ -\frac{\kappa^3}{8} e \\ -\frac{\kappa^3}{8} \eta_{d}  x_{:,{d}} \\ -\frac{\kappa^3}{16} (\eta_{d}  x_{:,{d}})^2 + \frac{\kappa}{4} e\\ -\frac{\kappa^3}{48}(\eta_{d}  x_{:,{d}})^3 + \frac{\kappa}{4} \eta_{d}  x_{:,{d}} + \frac{1}{2} e }.
	\end{split}
\]

\subsection{Heaviside activation function} \label{A:HEAVI}
For each fixed $z\in \R$ and $\kappa \to \infty$
the logistic activation function $\sigma_\kappa(z)$ 
converges to the  Heaviside  function $\sigma_\infty(z)$ defined in \eqref{eq:HEAVI}.
However, as the interval of convergence $[-\frac{\pi}{\kappa}, \frac{\pi}{\kappa}]$
of the Taylor expansion~\eqref{eq: logistic AF - Taylor} shrinks with incresing $\kappa$,
we now derive a different kind of approximation of the Heaviside kernel function 
$\psi_\infty(x,\eta)=\sigma_\infty(x\cdot \eta)$ of the form \eqref{eq:FACKER}.
Our starting point is the Fourier epansion 
\[
 \sigma_\infty(z) = \frac{1}{2} + \frac{2}{\pi} \sum_{\ell=1}^\infty \frac{1}{2\ell-1} \sin \left( \frac{(2\ell-1) \pi}{T} z \right),
\]
which converges pointwise within the interval $[-T, T]$ for suitable $T>0$. 
Euler's formula   $\sin(z) = (\exp(\mathrm{i}z)-\exp(-\mathrm{i}z))/2\mathrm{i}$ yields
\[
 \sigma_\infty(z) = \frac{1}{2} + \frac{1}{\pi\mathrm{i}} \sum_{\ell=1}^\infty \frac{1}{2\ell-1} \left( \exp \left( \frac{(2\ell-1) \pi \mathrm{i}}{T} z \right) - \exp \left( -\frac{(2\ell-1) \pi \mathrm{i}}{T} z \right) \right).
\]
Inserting $z = x\cdot \eta$, we obtain the representation 
\[
 \psi_\infty(x, \eta) = \frac{1}{2} + \frac{1}{\pi\mathrm{i}} \sum_{\ell=1}^\infty \frac{1}{2\ell-1} \left( \prod_{k=1}^d \exp \left( \frac{(2\ell-1) \pi \mathrm{i}}{T} x_k  \eta_k \right) - \prod_{k=1}^d \exp \left( -\frac{(2\ell-1) \pi \mathrm{i}}{T} x_k  \eta_k\right) \right)
\]
of the Heaviside kernel function for $x\cdot \eta \in[-T,T]$.
Truncating this expansion at $\ell = \ell_0$, 
we obtain the desired factorization  \eqref{eq:FACKER} 
of the resulting approximation with $r= 2 \ell_0 + 1$ summands.
In this case, the  decomposition \eqref{eq:KERDEC} 
of the associated functional tensor $\mathbf{\Psi}$ reads
\[
\begin{split}
 \mathbf{\Psi} & =  \Delta \cdot \left[ \frac{1}{2} \bigotimes_{k=1}^d e ~+~ \frac{1}{\pi\mathrm{i}} \sum_{\ell=1}^{\ell_0} \frac{1}{2\ell-1}  \bigotimes_{k=1}^d \exp \left( \frac{(2\ell-1) \pi \mathrm{i}}{T} \eta_k x_{:,k} \right) \right. \\
 &\qquad \qquad \left. - ~\frac{1}{\pi\mathrm{i}} \sum_{\ell=1}^{\ell_0} \frac{1}{2\ell-1} \bigotimes_{k=1}^d \exp \left( -\frac{(2\ell-1) \pi \mathrm{i}}{T} \eta_k x_{:,k} \right)  \right].
\end{split}
\]

\bibliography{GelssIssagaliKhSIMODS_Revision}
\bibliographystyle{unsrturl}
\end{document}